\documentclass[12pt,twoside]{article}
\usepackage[english]{babel}
\usepackage[utf8]{inputenc}

\usepackage{amsthm,amsmath,amssymb, amsfonts,stmaryrd,multicol,dsfont,mathtools,faktor}
\usepackage[usenames,dvipsnames,svgnames,table]{xcolor}
\usepackage{subcaption}
\usepackage{tikz}
\usetikzlibrary{fit}
\usetikzlibrary{arrows.meta}

\usepackage[left=3cm,right=3cm,top=3cm,bottom=3cm]{geometry}

\usepackage[pdftex,pdfauthor={Laura Monk},pdftitle={Typical hyperbolic surfaces have an
  optimal spectral gap}]{hyperref} \hypersetup{colorlinks=true, breaklinks=true, urlcolor=
  blue,linkcolor= blue, citecolor=orange}

\usepackage[capitalize]{cleveref} \crefname{equation}{equation}{equations}
\Crefname{equation}{Equation}{Equations} \crefname{figure}{Figure}{Figures}
\Crefname{figure}{Figure}{Figures}

\theoremstyle{theorem}
\newtheorem{thm}{Theorem}[section]
\newtheorem{prp}[thm]{Proposition}

\newtheorem{lem}[thm]{Lemma}

\theoremstyle{remark}
\newtheorem{ques}[thm]{Question}
\newtheorem{rem}[thm]{Remark}
\theoremstyle{definition}
\newtheorem{defi}[thm]{Definition}
\newtheorem{nota}[thm]{Notation}

\theoremstyle{example}
\newtheorem{expl}[thm]{Example}

\title{\textbf{Typical hyperbolic surfaces have an optimal spectral gap}}
\author{Laura Monk}
\date{Current Developments in Mathematics 2025}

\newcommand{\R}{\mathbb{R}}
\newcommand{\C}{\mathbb{C}}
\newcommand{\Z}{\mathbb{Z}}
\newcommand{\IH}{\mathbb{H}}

\DeclareMathOperator{\diam}{diam}
\newcommand{\PSL}{\mathrm{PSL}_2(\mathbb{R})}
\def\Tr{{\rm Tr}}
\renewcommand{\d}{\, \mathrm{d}}
\newcommand{\x}{\mathbf{x}}

\let\div\relax
\newcommand{\div}[1]{\left(\frac{#1}{2}\right)}

\newcommand{\cT}{\mathcal{T}}
\newcommand{\cR}{\mathcal{R}}
\newcommand{\cF}{\mathcal{F}}
\newcommand{\cM}{\mathcal{M}}
\newcommand{\cL}{\mathcal{L}}
\newcommand{\cP}{\mathcal{P}}

\newcommand{\cN}{\mathcal{N}}
\newcommand{\cG}{\mathcal{G}}
\newcommand{\rK}{\mathrm{K}}
\newcommand{\rN}{\mathrm{N}}
\newcommand{\cE}{\mathcal{E}}

\renewcommand{\O}[2][ ]{\mathcal{O}_{#1} \left( #2 \right)}

\newcommand{\MCG}{\mathrm{MCG}}
\newcommand{\Sf}{\mathbf{S}}
\newcommand{\g}{g_{\mathbf{S}}}
\newcommand{\n}{n_{\mathbf{S}}}
\newcommand{\domain}{\mathfrak{D}}
\newcommand{\cc}{\mathfrak{n}}

\newcommand{\cop}{{\mathbf{c}}_{\mathrm{op}}}
\newcommand{\copg}{{\bar{\mathbf{c}}}_{\mathrm{op}}}
\newcommand{\cI}{\mathcal{I}}
\newcommand{\eqc}[1]{[ #1 ]_{\mathrm{loc}}}

\newcommand{\Volwp}[1][g]{\mathrm{Vol}_{#1}^{\mathrm{\scriptsize{WP}}}}
\newcommand{\Pwpo}{\mathbb{P}_g^{\mathrm{\scriptsize{WP}}}}
\newcommand{\Ewpo}[1][g]{\mathbb{E}_{#1}^{\mathrm{\scriptsize{WP}}}}

\newcommand{\Pwp}[1]{\Pwpo \left( #1 \right)}
\newcommand{\Ewp}[2][g]{\mathbb{E}_{#1}^{\mathrm{\scriptsize{WP}}} \Bigg[ #2 \Bigg]}

\newcommand{\curve}{\mathbf{c}}
\newcommand{\type}{\mathbf{T}}
\newcommand{\av}[2][\mathrm{all}]{\langle #2 \rangle_g^{{#1}}}
\newcommand{\ord}{K}

\DeclareMathOperator{\hyp}{hyp}
\newcommand{\thetaNe}{\vec{\theta}_{\mathrm{ne}}}
\newcommand{\thetaAc}{\vec{\theta}_{\mathrm{ac}}}

\newcommand{\Lambdabeta}{\Lambda^\beta}

\newcommand{\Lambdain}{\Lambda_{\mathrm{in}}}

\DeclareMathOperator\argch{argcosh}
\makeatletter
\newcommand{\smallbullet}{} % for safety
\DeclareRobustCommand\smallbullet{%
  \mathord{\mathpalette\smallbullet@{0.5}}%
}
\newcommand{\smallbullet@}[2]{%
  \, \vcenter{\hbox{\scalebox{#2}{$\m@th#1\bullet$}}} \,%
}
\makeatother

\begin{document}

\maketitle

\abstract{The first non-zero Laplace eigenvalue of a hyperbolic surface, or its spectral gap,
  measures how well-connected the surface is: surfaces with a large spectral gap are hard to cut in
  pieces, have a small diameter and fast mixing times. For large hyperbolic surfaces (of large area
  or large genus $g$, equivalently), we know that the spectral gap is asymptotically bounded above
  by $\frac 14$. The aim of these talks is to present joint work with Nalini Anantharaman, where we
  prove that most hyperbolic surfaces have a near-optimal spectral gap. That is to say, we prove
  that, for any $\epsilon > 0$, the Weil--Petersson probability for a hyperbolic surface of genus
  $g$ to have a spectral gap greater than $\frac 14- \epsilon$ goes to one as $g$ goes to
  infinity. This statement is analogous to Alon’s 1986 conjecture for regular graphs, proven by
  Friedman in 2003. I will present our approach, which shares many similarities with Friedman’s
  work, and introduce new tools and ideas that we have developed in order to tackle this problem.}

\begingroup
\hypersetup{linkcolor=black}
\tableofcontents
\endgroup

\section{Introduction: the spectral gap of a compact hyperbolic surface}
\label{sec:intro}

\subsection{Hyperbolic surfaces and the high-genus limit}

The objects we will study throughout these lectures are \emph{compact hyperbolic surfaces of large
  genus}. For a integer $g \geq 2$, by the classification of surfaces, we know that there exists one
and only one oriented, connected, closed (i.e. compact without boundary) surface of genus $g$; we
denote as $S_g$ such a surface.  We are interested in \emph{hyperbolic surfaces}, i.e. Riemannian
metrics on the topological surface $S_g$ of constant curvature $-1$. Such a hyperbolic surface is
always isometric to a quotient $X=\Gamma \diagdown \IH$ of the hyperbolic plane
\begin{equation}
  \label{eq:hyp_plane}
  \IH
  = \{ x + iy, y > 0 \} 
  \quad
  \text{equipped with}
  \quad
  \d s^2 = \frac{\d x^2 + \d y^2}{y^2}
\end{equation}
by a Fuchsian group $\Gamma$, a discrete, co-compact subgroup of $\PSL$. 
The set of all hyperbolic metrics on $S_g$, up to isometry, is the \emph{moduli space} $\cM_g$.

By the Gauss--Bonnet theorem, the area of a hyperbolic surface of genus $g$ is always
$2 \pi (2g-2)$. We will be interested in the large-genus regime, which can also be interpreted as a
large-scale regime.

One of the main characters that will appear in our story are \emph{closed geodesics}, and their
lengths.  A loop $\curve$ on the surface $X$ is a piece-wise smooth map
$\curve : \R \diagup \Z \rightarrow X$ with nowhere vanishing derivatives. Note that, in particular,
all of our loops are \emph{oriented}. Two loops $\curve_0$, $\curve_1$ are said to be homotopic if
there exists a continuous map $[0,1]\times \R \diagup \Z \rightarrow X$ which matches with
$\curve_0$ and $\curve_1$ when the first variable is fixed to be $0$ and $1$ respectively.  A loop
is called a \emph{curve} if it is simple (no self-intersections) and non-contractible. We extend
these definitions to \emph{multi-loops} $\curve = (\curve_1, \ldots, \curve_\cc)$ and multi-curves
(in the latter case we asssume that the whole multi-loop has no self-intersection, and that for
$i \neq j$, $\curve_i$ is not homotopic to $\curve_j$ nor $\curve_j^{-1}$).

Due to the negative curvature, in each non-trivial homotopy class on $X$ admits a unique
length-minimizer, which is a closed geodesic. We denote as $\cG(X)$ the set of primitive closed
geodesics on $X$, i.e. closed geodesics which are not iterates of a closed geodesic. For
$\gamma \in \cG(X)$, we denote as $\ell_X(\gamma)$ the length of the closed geodesic $\gamma$ on
$X$. The set $\{\ell_X(\gamma), \gamma \in \cG(X)\}$ is called the \emph{length spectrum}; it
contains a lot of information on the hyperbolic surface $X$.

\subsection{The spectral gap of the Laplacian}

The core focus of these notes is the spectrum of the Laplace--Beltrami operator, or Laplacian. The
Laplacian is a positive unbounded operator on the space of square-integrable functions $L^2$. One of
its defining properties is that it is invariant by isometries. The spectrum of the Laplacian
$\Delta$ is defined as the set of complex numbers $\lambda$ such that the operator
$\Delta - \lambda \mathrm{Id}$ does not have a continuous inverse defined on $L^2$.

On the hyperbolic plane $\IH$, the Laplacian can be written simply in coordinates as
$\Delta = -y^2 \big(\frac{\partial^2}{\partial x^2}+\frac{\partial^2}{\partial y^2}\big)$. Its
spectrum is purely continuous and exactly equal to the interval $[\frac 14, + \infty)$.  This
special value $\frac 14$, central to this discussion, arises from the observation that, for any
smooth $L^2$ function $f$ on $\IH$, we have
$$\int_\IH f \Delta f \geq \frac 14 \int_\IH f^2$$ by a simple application of Cauchy--Schwarz and
integration by parts. In other words, the operator $\Delta - \frac 14 \mathrm{Id}$ is positive on
$L^2(\IH)$ \cite{mckean1970}.

Considering the Laplacian on a compact hyperbolic surface $X = \Gamma \diagdown \IH$ corresponds to
restricting the domain of the Laplacian on $\IH$ to $\Gamma$-periodic functions.  Because the
surfaces we study are compact, the spectrum in this case is a discrete sequence of eigenvalues
\begin{equation}
  \label{eq:sp_lp}
  0 = \lambda_0 < \lambda_1 \leq \ldots \leq
  \lambda_n \underset{n \rightarrow + \infty}{\longrightarrow} + \infty.
\end{equation}
The first eigenvalue, $\lambda_0=0$, corresponds to constant eigenfunctions (which belong to
$L^2(X)$ because the surface has finite area). Its multiplicity is one because the dimension of the
space of constant functions is the number of connected components, here, one. The \emph{spectral
  gap} is the first non-trivial eigenvalue, $\lambda_1 > 0$.

It is very easy to construct examples of hyperbolic surfaces of any genus with a small $\lambda_1$
by pinching a separating closed geodesic, as represented in \cref{fig:dumbell}. This is what we
often refer to as a \emph{dumbbell} in spectral geometry. The idea is that, at the limit, the surface
has two connected components, and hence $\lambda_1$ goes to $0$ as the length $\epsilon$ of the
pinched closed geodesic goes to $0$.

\begin{figure}[h]
  \centering
  \includegraphics{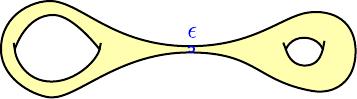}
  \caption{The dumbbell example.}
  \label{fig:dumbell}
\end{figure}

The dumbbell surface is poorly connected (it is easy to disconnect it by cutting it with
scissors). Its diameter goes to infinity as $\epsilon$ goes to $0$, and, hence, some points in the
surface become very far apart. Hyperbolic surfaces are interesting objects to study dynamically as
they are simple examples of chaotic systems, due to their negative curvature. For a small
$\epsilon$, the surface has a bottle-neck, which means that particles or random walks starting on
one side of the dumbbell will take a very long time to travel to the other side; this means that the
mixing time is unusually long.

All of these observations are actually characterizations of surfaces with a small spectral gap, and
we have the following relationships between all these notions (see the references for precise statements).
\begin{figure}[h]
  \centering
  \begin{tikzpicture}
    \node[thick, dotted,text centered,draw, rounded corners] (lambda) at (0,0)
    {$\lambda_1(X)$ is big};
    \node[thick, dotted,text centered,draw, rounded corners] (connected) at (5,-1)
    {$X$ is well-connected};
    \node[thick, dotted,text centered,draw, rounded corners] (diam) at (-4,-1)
    {$\diam(X)$ is small};
    \node[thick, dotted,text centered,draw, rounded corners] (dyn) at (0,-2)
    {$X$ has good dynamical properties};
    \draw[->, thick, shorten <= 5pt, shorten >= 5pt] (lambda) -- (diam)
    node [midway,above left] {\cite{magee2020b}};
    \draw[<->, thick, shorten <= 5pt, shorten >= 10pt] (lambda) -- (connected)
    node [midway,above right] {\cite{buser1982,cheeger1970}};
    \draw[<->, thick, shorten <= 5pt, shorten >= 5pt] (lambda) -- (dyn)
    node [midway,right] {\cite{ratner1987,golubev2019}};
    \end{tikzpicture}
  \caption{Relationship between the spectral gap and other properties of the surface.}
  \label{fig:eq_gap}
\end{figure}
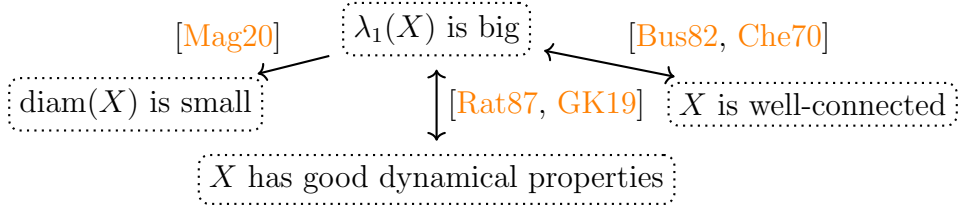

\subsection{Surfaces with a large spectral gap}
\label{sec:surfaces-with-large-gap}

Now that we know that having a large spectral gap is a desirable property, our next question is: how
big can a spectral gap be? We believe that the hyperbolic surface with the largest spectral gap is
the Bolza surface, obtained by gluing the opposite sides of a regular hyperbolic octagon with all
angles equal to $\pi/4$, as represented in Figure \ref{fig:bolza}. For this surface, we have a very
precise numerical approximation by Strohmaier and Uski \cite{strohmaier2013},
\begin{equation*}
  \lambda_1
  \approx 3.8388872588421995185866224504354645970819150157.
\end{equation*}
Whilst the fact that this maximises $\lambda_1$ over all hyperbolic surfaces is still a conjecture
to this day, recent developments have reached extremely close upper bounds, less than $2.10^{-5}$
above this value \cite{kravchuk2024,bonifacio2022}.

\begin{figure}[h]
  \centering
  \includegraphics{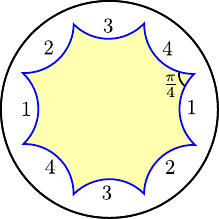}
  \caption{The construction of the Bolza surface.}
  \label{fig:bolza}
\end{figure}

However, the Bolza surface has genus $g=2$ and we are interested in the large-genus regime. As the
genus $g$ goes to infinity, one can argue that, because the surface ``grows'', it will start
``looking like'' the hyperbolic plane $\IH$ and, in particular, the entire interval $[\frac 14, +
\infty)$ will be populated by eigenvalues.
More precisely, Huber proved in \cite{huber1974} that
\begin{equation}
  \label{eq:huber}
  \limsup_{g \rightarrow + \infty} \, \sup \{\lambda_1(X), X \in \mathcal{M}_g\} \leq \frac 14
\end{equation}
or, equivalently, $\lambda_1 \leq \frac 14 + o(1)$ as $g \rightarrow + \infty$. This means that the
value $\frac 14$ is asymptotically the largest possible spectral gap. Buser conjectured in
\cite{buser1984} that \eqref{eq:huber} is an equality, meaning there exists a sequence
$(X_n)_{n \geq 1}$ of compact hyperbolic surfaces of genus going to $+ \infty$ and such that
$\lim_{n \rightarrow + \infty} \lambda_1(X_n)=\frac 14$.  This very popular open question was
answered positively recently by Hide and Magee in a celebrated article \cite{hide2023a} using random
covers.

\subsection{Typical hyperbolic surfaces of high genus}

The purpose of these lectures is to present a proof of the following statement, the main result from
my two-part joint work with Nalini Anantharaman \cite{anantharaman2023,anantharaman2025}.

\begin{thm}
  \label{thm:dream}
  For any $\epsilon>0$, 
  \begin{equation*}
    \lim_{g \rightarrow + \infty} \Pwp{\lambda_1 \geq \frac 14 - \epsilon} = 1.
  \end{equation*}
\end{thm}
Here, $\Pwpo$ denotes the Weil--Petersson probability measure on the moduli space $\cM_g$, a
natural, continuous model popularized by Mirzakhani, which describes typical hyperbolic surfaces of
fixed genus $g$.  As a consequence, \cref{thm:dream} tells us that hyperbolic surfaces with a
near-optimal spectral gap not only exist: they are the typical behavior in the large genus limit.

This is the conclusion of a sequence of increasingly precise partial results which are listed in
\cref{tab:gaps}. It is worth pointing out that Mirzakhani's first uniform spectral gap was obtained
in \cite{mirzakhani2013} by quite a loose estimate using the Cheeger constant. All other results,
starting from the significant jump to $\frac{3}{16}-\epsilon$ by the two independent teams Wu--Xue
\cite{wu2022} and Lipnowski--Wright \cite{lipnowski2024}, use the same approach, called the
\emph{trace method}. The idea behind the trace method is discussed in \cref{sec:trace}, and we will
have the opportunity to discuss the reason behind these iterative improvements leading to the optimal
result. 

\def\arraystretch{1.5}
\begin{table}[h]
  \centering
  \begin{tabular}{|l|l|l|}
    \hline
    Gap obtained & Comparison with $\frac 14$ & Reference \\
    \hline
    $\lambda_1 \geq \frac 14 \big(\frac{\log 2}{\log 2+\pi}\big)^2
    $
    & $ \big(\frac{\log 2}{\log 2+\pi}\big)^2 \approx 0.03267$
    & \cite{mirzakhani2013} \\
    \hline
    $\lambda_1 \geq \frac{3}{16}-\epsilon$
    & $\frac{3}{16} = \frac 14 \times \frac 34$
    & \cite{wu2022} and \cite{lipnowski2024}  \\
    \hline
    $\lambda_1 \geq \frac 29 - \epsilon$
    & $\frac{2}{9} = \frac 14 \times \frac 89$  
                                              & \cite{anantharaman2023}\\
    \hline
    $\lambda_1 \geq \frac 14 - \epsilon$
                 & Optimal.
                                              & \cite{anantharaman2025}\\
    \hline
  \end{tabular}
  \caption{History of the spectral gap problem for random compact hyperbolic surfaces of high genus
    sampled with the Weil--Petersson probability measure.}
  \label{tab:gaps}
\end{table}

\begin{rem}[The Selberg $\frac 14$ conjecture]
  This exposition would not be complete without a quick mention of one of the biggest open problems
  in the spectral geometry of hyperbolic surfaces and arithmetics: the Selberg $\frac 14$
  conjecture. This is a question about \emph{congruence covers}, hyperbolic surfaces realized as
  the quotient of $\IH$ by the groups
  \begin{equation}
    \label{eq:cong}
    \Gamma(N) = \{ M \in \mathrm{PSL}_2(\Z) \, : \, M \equiv \mathrm{Id} \, (\text{mod } N)\}.
  \end{equation}
 For $N \geq 2$, the quotient $X(N) = \Gamma(N) \diagdown \IH$ is a finite-area
  non-compact hyperbolic surface, of genus behaving like $N^3$. Selberg observed in
  \cite{selberg1965} that, if one were to prove that congruence surfaces have a certain spectral
  gap, they could then obtain interesting cancellations in certain arithmetic sums. This motivated
  the now famous Selberg $\frac 14$ conjecture, first stated in \cite{selberg1965}:
  \begin{equation}
    \label{eq:selb_14}
    \forall N \geq 1, \quad \lambda_1(X(N)) \geq \frac 14.
  \end{equation}
  This is the analogue of the Riemann hypothesis for the Selberg zeta function. In his article,
  Selberg obtains the partial result $\lambda_1(X(N)) \geq \frac{3}{16}$ (my understanding is that
  those two $\frac{3}{16}$, from arithmetic v.s. random surfaces, are distinct). It is very
  interesting to me that, like in the case of random hyperbolic surfaces, any improvement on this
  bound in the spectral gap has implications in terms of cancellations in certain quantities:
  arithmetic sums for Selberg, and geometric averages for us. The current best know bound is due to
  Kim and Sarnak, who proved that $\lambda_1(N) \geq \frac 14 \left(1-\frac{49}{1024}\right)$ in
  \cite{kim2003}. Compared to the case of random hyperbolic surfaces, it seems extremely hard to
  find a cancellation mechanism that works all the way to the optimal gap $\frac 14$.

  Note that congruence covers are a countable set and the Weil--Petersson measure has no atoms, so
  \cref{thm:dream} implies nothing in the setting of the Selberg~$\frac 14$ conjecture. In a sense,
  these two approaches are orthogonal: we are studying random hyperbolic surfaces, averaging over all
  possible metrics and viewing the lengths of geodesics as somehow uniformly distributed; to the
  contrary, congruence covers are very rigid objects with high symmetries.
\end{rem}

\subsection{Link with $d$-regular graphs}
\label{sec:link-d-reg}

From the beginning of this research project, my collaborator Nalini Anantharaman and myself have
been drawing some of our inspiration from the world of $d$-regular graphs. A $d$-regular graph is a
graph such that every vertex has exactly $d$ neighbors.  These discrete objects, which have been
extensively studied in graph theory, behave in a remarkably similar way to hyperbolic
surfaces. Indeed, they are homogeneous (a small neighborhood of a point is always the same), and the
size of their balls grows exponentially fast in the radius. A little dictionary between these two
settings is provided in \cref{tab:comp_graph}. Note that, throughout these notes, we assume for the
sake of simplicity of the exposition that $G$ is not bipartite\footnote{A bipartite graph is a graph
  with vertice set which can be split into $V=V_1 \sqcup V_2$ so that each edge has one endpoint in
  $V_1$ and one in $V_2$; such a graph has an additional trivial eigenvalue $-d$ which has very
  little consequences to our discussions, but can be a bit cumbersome for writing.}. As in the case
of hyperbolic surfaces, we will be interested in the large-scale regime, i.e. we will consider that
the number $n$ of vertices of the graph is large.

\begin{table}[h]
  \centering
  \begin{tabular}{|c|c|c|}
    \hline
    Object & hyperbolic surface $X$ & $d$-regular graph $G=(V,E)$ \\
    \hline
    Size & genus $g \rightarrow + \infty$ & number of vertices $n \rightarrow + \infty$ \\
    \hline
    Local geometry & hyperbolic balls & each vertex has $d$ neighbors \\
    Size of $R$-balls
         & $\approx e^R$ & $\approx (d-1)^R$ \\
    \hline
    Operator
           & Laplacian $\Delta$
             & Adjacency matrix $A$ \\
    Spectrum & $0 = \lambda_0 < \lambda_1 \leq \ldots \rightarrow + \infty$
                                          & $d = \lambda_1 > \lambda_2 \geq \ldots \geq \lambda_n
                                            > -d$ \\
    Trivial eigenvalue
         & $\lambda_0=0$
                                    & $\lambda_1=d$\\
    Spectral gap & $\lambda_1$ & $d-\lambda_+, \lambda_+=\max \{\lambda_2,-\lambda_n\}$ \\
    \hline
    Universal cover & hyperbolic plane $\IH$ & $d$-regular tree \\
    and its spectrum & $[\frac 14, + \infty)$
                                           & $[-2\sqrt{d-1},2\sqrt{d-1}]$\\
    \hline
  \end{tabular}
  \caption{The comparison between hyperbolic surfaces and regular graphs.}
  \label{tab:comp_graph}
\end{table}

For a non-bipartite $d$-regular graph $G=(V,E)$ with $n$ vertices, we can study the spectrum
\begin{equation}
  \label{eq:sp_A}
  d = \lambda_1 > \lambda_2 \geq \ldots \geq \lambda_n > -d
\end{equation}
of the adjacency matrix (note that the adjacency matrix differs from the Laplacian matrix by
$d \, \mathrm{Id}$, because of the regularity of the graph, so they are the same objects). The
trivial eigenvalues is $\lambda_1=d$, corresponding to the constant eigenfunction. The equivalent of
the spectral gap is the spacing between the trivial eigenvalue $d$ and
$\lambda_+=\max\{\lambda_2,-\lambda_n\}$.

The equivalent of Huber's bound \eqref{eq:huber} for $d$-regular graphs is the Alon--Boppana bound
\cite{nilli1991,alon1986}, which states that $\lambda_2 \geq 2 \sqrt{d-1}-o(1)$ as
$n \rightarrow + \infty$. The value $2 \sqrt{d-1}$ plays the exact role of $\frac 14$ as it is the
edge of the spectrum of the universal cover of all $d$-regular graphs, the $d$-regular tree. Alon
famously conjectured in \cite{alon1986} that \emph{most} large $d$-regular graphs have a
near-optimal spectral gap. This was proven many years later by Friedman in a 120-pages proof
\cite{friedman2003}. There now exists several modern proofs, the first one by
Bordenave~\cite{bordenave2020}, and very recently, a completely new approach by Chen, Garza-Vargas,
Tropp and Van Handel \cite{chen2025}. Shortly after these lectures were given, this new approach has
been further extended by Magee, Puder and Van Handel \cite{magee2025}, to prove that random covers
of a compact hyperbolic surface typically have a near-optimal spectral gap in the high-degree limit.

In our attempts at proving \cref{thm:dream}, we have naturally converged to Friedman's method,
without first realizing it. Discovering that we were on a similar track shed a new light on
Friedman's work and allowed us to understand and appreciate it deeply; for this reason, we named
some of the objects introduced in this project in his honor. Some elements of Friedman's proof and
their relationship with ours will be discussed in \cref{sec:trace}.

\subsection{Open questions}
\label{sec:open-questions}

Whilst we cannot improve our result by replacing $\frac 14$ by a larger constant due to Huber's
result, there are a lot of finer interesting questions -- some of which can be expected to be very hard.

\begin{ques}
  Find a sequence $\epsilon_g \rightarrow 0^+$ such that
  \begin{equation}
  \lim_{g \rightarrow + \infty} \Pwp{\lambda_1 \geq \frac 14 - \epsilon_g}=1.    
  \end{equation}
  For which scaling regime (i.e. for which behavior of $\epsilon_g$ as a function of $g$) does this
  cease to be true?  Same question with $\Pwp{\lambda_1 \geq \frac 14 + \epsilon_g}$: for which
  scaling is this probability zero or non-zero?
\end{ques}

\begin{ques}
  \label{ques:TW}
  Prove that there is a number $p \in (0,1)$ such that
  \begin{equation}
    \label{eq:lambda14}
    \lim_{g \rightarrow + \infty} \Pwp{\lambda_1 \geq \frac 14} = p.
  \end{equation}
  What is the value of $p$?
\end{ques}

Both of these questions are inspired by the intuition that the statistics of the eigenvalues at the
edge of the bulk spectrum follow the Tracy--Widom distribution. Proving such a statement would be
extremely exciting but feels quite out of reach at this moment in time, by comparison with the world
of regular graphs. Indeed, \cref{ques:TW} was answered positively very recently, with a value
$p \approx 69\%$, by Huang, McKenzie and Yau \cite{huang2025}. The reason why these developments
come at a much later time than Friedman's proof is that the trace method is way too soft an approach
to go anywhere near this degree of precision.  Indeed, getting in the bowels of our proof and
expliciting everything, one might be able to obtain a $\epsilon_g$ decaying like a negative power of
$\log(g)$. However, by comparison with random matrix theory and the case of regular graphs, the
interesting scale in actually expected to be a negative power of $g$. The breakthroughs of Huang,
McKenzie and Yau rely on exiting new findings in the theory of Green functions, which do not, a
priori, seem to apply to hyperbolic surfaces.

\subsection{Organization of these lecture notes}

The lecture notes are organized in a structure compatible with the two lectures.  The aim of the
first talk is to introduce and motivate the result, and then discuss the notion of
\emph{Friedman--Ramanujan function} and its the stability by convolution; i.e., the contents of
Sections \ref{sec:intro} and \ref{sec:trace}. The second talk will then be entirely focused on the
study of the \emph{volume functions} which describe the distribution of the length of closed
geodesics on random hyperbolic surfaces; this corresponds to Section \ref{sec:volume}. % In the last
% section, additional elements which I will not have time to cover in depth are provided, related to
% the last big block of the proof: the removal of \emph{tangled surfaces} using our Moebius inversion
% formula.

\section{The trace method: regular graphs and hyperbolic surfaces}
\label{sec:trace}

Because the thought process behind our trace method and our cancellation argument has been presented
in details in \cite{anantharaman2023} as well as our exposition article \cite{anantharaman2024}, the
full explanation is not reproduced here, and we rather present the ideas a bit differently.
This will hopefully allow for us to spend more time on other aspects of the proof, which are new to
\cite{anantharaman2025}. 

\subsection{The trace method for regular graphs}
\label{sec:baby-trace-method}

Let $G = (V,E)$ be a non-bipartite $d$-regular graph with $n$ vertices. Recall that the spectrum of
the adjacency matrix $A$ is denoted as
\begin{equation}
  \label{eq:sp_G}
  d = \lambda_1 > \lambda_2 \geq \ldots \geq \lambda_n > -d
\end{equation}
and we define the spectral gap of $G$ to be $d-\lambda_+$, where
$\lambda_+=\max(\lambda_2,-\lambda_n)$.

\subsubsection{The trace method}
\label{sec:simple-trace-method}

The general idea behind trace methods is to relate the spectrum $(\lambda_j)_{1 \leq j \leq n}$ of
the graph $G$ to its geometry, e.g. its closed paths and their lengths. A simple way to do so, which
I find very enlightening, consists in picking an integer $\ell > 0$ and observing that:
\begin{equation}
  \label{eq:tr_G}
  \Tr(A^\ell) = \sum_{j=1}^n \lambda_j^\ell
  = \sum_{v \in V} \# \{\text{loops of length } \ell \text{ based at } v\}.
\end{equation}
The first equality is a trivial consequence of the fact that the matrix $A$ is diagonalizable.
The second equality is obtained by observing that the entries of the matrix $A^\ell$ exactly count
the number of paths of length $\ell$ between pairs of vertices. Because the trace sums over diagonal
elements, the paths we are counting are actually \emph{loops}, i.e. closed paths.

We can deduce from \eqref{eq:sp_G} and \eqref{eq:tr_G} the following bound true for the
loop-counting function
\begin{equation}
  \label{eq:as_G_1}
  \Big|\sum_{v \in V} \# \{\text{loops of length } \ell \text{ based at } v\}
  - d^\ell \Big| \leq n \lambda_+^\ell
\end{equation}
which provides us with its asymptotic behavior as $\ell \gg 1$.  In other words, the leading-order
term for the geometric count of loops is of the form $d^\ell$, and related to the contribution of
the trivial eigenvalue $\lambda_1=d$. The spectral gap is then visible in the spacing between the
growth of the leading and sub-leading terms. In particular, if I were to find a constant $\rho>0$
such that
\begin{equation}
  \label{eq:as_G_2}
  \Big|\sum_{v \in V} \# \{\text{loops of length } \ell \text{ based at } v\}
  - d^\ell \Big| \ll \rho^\ell
\end{equation}
as $n \rightarrow + \infty$, then I could deduce that $\lambda_+ \leq \rho$ or, in
other words, I would get a lower bound on the spectral gap of my graph. The quality of the spectral
gap is related to the size of the $\rho$ we achieve: the smaller the $\rho$ the tighter the estimate
and the better the gap.

\subsubsection{Ramanujan functions}
\label{sec:ramanujan-functions}

Let us now get improve this trace method, and discuss a few ideas present in Friedman's proof of the
Alon conjecture, namely that, for any $\epsilon >0$, typical large $d$-regular graphs satisfy
$\lambda_+ \leq 2 \sqrt{d-1}+\epsilon$.

The geometric quantity that Friedman counts is \emph{irreducible loops}, i.e. closed paths on the
graph $G$ which never go back-and-forth along an edge in two consecutive steps. We define the
counting function
\begin{equation*}
  N_G^{\mathrm{irr}}(\ell) := \# \{ \text{irreducible loops of length } \ell \text{ on } G\}.
\end{equation*}
Because, at each step of the path, we now only have $d-1$ options for the next step, instead of $d$,
the growth-rate of this quantity is $(d-1)^\ell$ rather than $d^\ell$.

Now, the heuristics above can be made precise in the following manner. Friedman defined the
following class of functions, which he calls Ramanujan functions, and should be reminiscent of
equation \eqref{eq:as_G_2}. 

\begin{defi}[{\cite{friedman2003}}]
  A function $f : \Z_{>0} \rightarrow \C$ is said to be a $d$-\emph{Ramanujan function} if there exists
  a polynomial function $p$ and a constant $c$ such that
  \begin{equation*}
    \forall \ell \in \Z_{>0}, \quad
    |f(\ell) - p(\ell) (d-1)^\ell| \leq c \, \ell^c (d-1)^{\ell/2}.
  \end{equation*}
\end{defi}

Lubotzky, Phillips and Sarnak observed in \cite{lubotzky1988} that, if the counting function
$N_G^{\mathrm{irr}}(\ell)$ is $d$-Ramanujan, then $\lambda_+ \leq 2 \sqrt{d-1}$. Such graphs, of
optimal spectral gap, are called \emph{Ramanujan graphs} -- this motivates the choice of the name.

The term $p(\ell) (d-1)^\ell$ is called the \emph{principal term}; it is the contribution we
naturally expect from the trivial eigenvalue $\lambda_1=d$. The crucial feature in this estimate is
the \emph{gap} in the exponents, from $(d-1)^\ell$ to $(d-1)^{\ell/2}$. It is in the size of this
gap that the optimal spectral gap result is created: a stronger estimate could not be true for large
$n$, whilst a weaker bound would only allow to obtain a sub-optimal spectral gap.

Notice that the requirement to be a Ramanujan function is a very strong requirement: we ask for a
very specific expansion, with a big leading term of a prescribed nature, and a much, much smaller
remainder. The principal term will always cloud our estimates and, in the following, a lot of
 effort is dedicated to dealing with the pollution generated by the trivial eigenvalue.

\subsubsection{Adding probabilities to the mix}
\label{sec:prob-trace-meth}

This is all very good, but there is \emph{no way} we could prove that the counting function
$N_G^{\mathrm{irr}}(\ell)$ is $d$-Ramanujan for a fixed, general $d$-regular graph $G$ (contrarily
to Lubotzky, Phillips and Sarnak, who were working with specific families of graphs with algebraic
structure in \cite{lubotzky1988}). We now need to use our probabilistic viewpoint and averaging
techniques. We introduce the \emph{average counting function}
\begin{equation*}
  N_{d,n}^{\mathrm{irr}}(\ell) := \mathbb{E}_{d,n}[N^{\mathrm{irr}}_G(\ell)]
\end{equation*}
where the expectation is defined with respect to a model for random $d$-regular graphs with $n$
vertices, which we shall not specify.

By an application of Markov's inequality, if one were to prove that the average counting function
$N^{\mathrm{irr}}_{d,n}(\ell)$ is a $d$-Ramanujan function, one could deduce that typical
$d$-regular graphs have a near-optimal spectral gap.

Going from a single counting function to an averaged version is a very powerful move. In this new
setting, the role of the large-scale asymptotic $n \rightarrow + \infty$ becomes much clearer: we
need to study the behavior of the functions $N^{\mathrm{irr}}_{d,n}(\ell)$ for $d$ fixed and
$n \gg 1$. Friedman proved, for his models of random $d$-regular graphs, that we have an asymptotic
expansion of the form
\begin{equation*}
  N^{\mathrm{irr}}_{d,n}(\ell)
  = \sum_{k=0}^{\ord} \frac{f_{d,k}(\ell)}{n^k} + \O[d,K]{\frac{\ell^{4K+2}(d-1)^{\ell}}{n^{K+1}}}.
\end{equation*}
The spectral gap question is then intimately related to the coefficients
$(f_{d,k}(\ell))_{k \geq 0}$ and their behavior as a function of $\ell$. For instance, in
\cite{friedman1991}, Friedman proved that, for $k \leq \sqrt{d-1}/2$, the function $f_{d,k}(\ell)$
is $d$-Ramanujan. He then deduces that $\lambda_+ \leq 2 \sqrt{d-1}+2\log(d)+C$ typically, for a
universal constant $C$.
Reaching the optimal spectral gap and proving that $\lambda_+\leq 2 \sqrt{d-1} + \epsilon$ requires
understanding \emph{all} of the coefficients $(f_{d,k}(\ell))_{k \geq 0}$, i.e. estimating the
counting function $N_{d,n}^{\mathrm{irr}}(\ell)$ with arbitrary precision as $n \rightarrow + \infty$.

\subsubsection{Restriction to fixed loop topologies}
\label{sec:restr-fixed-loop}

So, now, we have a new goal! We ``just'' have to prove that, for all $k$, the function
$f_{d,k}(\ell)$ is a Friedman--Ramanujan function. Unfortunately... This is false. Friedman proves
in \cite[Theorem 2.12]{friedman2003} that, for his model of random $d$-regular graphs, there exists
$k$ such that $f_{d,k}(\ell)$ is \emph{not}
$d$-Ramanujan. % We will discuss this result, and its adaptation
% to the case of hyperbolic surfaces, in more details in the final section of these notes.
The proof speaks to the strength of the method: basically, if all $(f_{d,k}(\ell))_{k \geq 0}$ were
$d$-Ramanujan functions, then we could prove very good bounds on the probability of having a small
spectral gap, decaying fast as $n \rightarrow + \infty$. However, one can prove that the probability
of having a small spectral gap goes to zero somehow slowly, using poorly connected examples. This
contradicts the fact that the coefficients of the counting functions are all $d$-Ramanujan.

This is a very technical aspect of the proof, % , and we will come back to it in time
 presented here for a specific reason: to motivate the introduction of type-specific counting
functions. Indeed, one of the steps that need to be taken to solve the issue foreshadowed above is
to split the counting function by \emph{types}:
\begin{equation}
  N_{d,n}^{\mathrm{irr}}(\ell) = \sum_{\type} N_{d,n}^\type(\ell)  
\end{equation}
where now $\type$ denotes a ``type'' of irreducible loop, and $N_{d,n}^\type(\ell)$ counts loops of
this type. Examples of types include simple (loops with no self-intersection), or the figure-eight
(loops with exactly one self-intersection).

These restricted counting functions are much better-behaved than the overall counting function
$N_{d,n}^{\mathrm{irr}}(\ell)$: Friedman proves that all of the terms of their asymptotic expansions
in powers of $1/n$ are, indeed, $d$-Ramanujan functions. Actually, the reason why the coefficients
of the overall counting function $N_{d,n}^{\mathrm{irr}}(\ell)$ are not $d$-Ramanujan function is an
exponential proliferation of the number of types $\type$ in the sum above, related to the surfaces
with a small spectral gap. So, whilst each term admits the correct asymptotic expansion, the
number of terms grows too fast to sum them.

\subsection{The heroes of our trace method}
\label{sec:heroes-our-trace}

Let us now return to the focus of these lectures, compact hyperbolic surfaces. All the aspects
presented in \cref{sec:baby-trace-method} are shared between $d$-regular graphs and hyperbolic
surfaces -- the definitions of the objects being easier in the graph case due to the discrete
setting. I will ask you to accept a little jump here, and trust that the discussion of the trace
method can be adapted (more details can be found in \cite{anantharaman2023,anantharaman2024}). The
trace formula we use is the Selberg trace formula~\cite{selberg1956}, which reads
\begin{equation}
  \label{eq:tr_selb}
  \sum_{j=0}^{+ \infty} \hat{h}(r_j)
  = (g-1) \int_\R r \hat{h}(r)  \tanh(\pi r) \d r
  + \sum_{\gamma \in \cG(X)}
  \sum_{k=1}^{+ \infty}
  \frac{\ell_X(\gamma) \, h(k \ell_X(\gamma))}{2 \sinh \div{k \ell_X(\gamma)}}
\end{equation}
where $h$ is an appropriate test function, $\hat{h}$ its Fourier transform and $r_j$ is a complex
number such that $\lambda_j = \frac 14 + r_j^2$.  As in the case of graphs, this formula relates the
spectrum $(\lambda_j)_{j \geq 0}$ of the Laplacian to a geometric counting function, here, counting
primitive closed geodesics $\gamma \in \cG(X)$.

The trace method sketched above for graphs translates here to a new objective: we wish to establish
a \emph{Ramanujan property} for certain \emph{average geometric counting functions}. Let us make
sense of these two notions, the heroes of our approach.

\subsubsection{Weil--Petersson volume functions}
\label{sec:volume-functions-1}

First, we introduce our volume functions, which will be the equivalent of the counting functions
$N_{d,n}^\type(\ell)$. The study of the volume functions is at the heart of our approach, and
arguably one of the significant elements of novelty in our work. We provided completely new closed
expressions and asymptotic results for them in \cite{anantharaman2023,anantharaman2025}, which we
will present in more detail in \cref{sec:volume}.

Let $g \geq 2$, and $\Volwp$ denote the Weil--Petersson measure on the moduli space
\begin{equation}
  \label{eq:moduli}
  \cM_g = \{\text{compact hyperbolic surfaces of genus } g\} \diagup \{\text{isometries}\}.
\end{equation}
The Weil--Petersson probability $\Pwpo$ is defined by renormalization of the measure $\Volwp$,
dividing by its total mass $V_g$. The associated expectation is
\begin{equation}
  \Ewpo[g] [\phi]:=\frac{1}{V_g} \int_{\cM_g} \phi(X) \d \Volwp(X).
\end{equation}

In \cite{anantharaman2023}, we introduced a ``counting'' function $V_g^{\mathrm{all}}(\ell)$ which
allows to study the distribution of the length spectrum on a random hyperbolic surface of
genus~$g$. % It corresponds to the function $N_{d,n}^{\mathrm{irr}}(\ell)$ in the case of $d$-regular
% graphs with $n$ vertices.

\begin{prp}[{\cite{anantharaman2023}}]
  For any $g \geq 2$, there exists a unique locally integrable function $V_g^{\mathrm{all}} :
  \R_{> 0} \rightarrow \R_{\geq 0}$ such that, for any test function $F$,
  \begin{equation}
    \av{F} := \Ewp{\sum_{\gamma \in \cG(X)} F(\ell_X(\gamma))}
    = \frac{1}{V_g} \int_0^{+ \infty} F(\ell) V_g^{\mathrm{all}}(\ell) \d \ell.
  \end{equation}
\end{prp}
\begin{rem}
  Throughout this text, by \emph{test function $F$}, we mean a measurable function
  $F: \R_{\geq 0} \rightarrow \C$ which is bounded and compactly supported.
\end{rem}
The overall volume function $V_g^{\mathrm{all}}(\ell)$ is the analogue of the overall counting
function $N_{d,n}^{\mathrm{irr}}(\ell)$ for random hyperbolic surfaces. It counts the number of
primitive closed geodesics of length $\ell$ on a random hyperbolic surface of genus $g$. 

As in the case of graphs, the function $V_g^{\mathrm{all}}(\ell)$ is too wild to study as is, and we
benefit greatly by decomposing it according to the topologies of the geodesics we wish to study:
\begin{equation}
  \label{eq:vall_type}
  V_g^{\mathrm{all}}(\ell) = \sum_{\type} V_g^\type(\ell)
\end{equation}
where the sum runs over all local types $\type$ of closed geodesics on a surface of
genus~$g$. Precise definitions of the notion of local types are postponed until
\cref{sec:volume}. For now, we can think of a local type as the topology of the geodesic itself (we
do not care about the way it is embedded in the surface of genus~$g$). Examples of local types are
the simple local type (geodesics with no self-intersections), and the figure-eight (geodesics with
exactly one self-intersection).

Any local type $\type$ has a absolute Euler characteristic $\chi(\type) \geq 0$, which is equal to
$0$ for the simple local type only.  Building on Mirzakhani's integration and topological recursion
formulae, we prove in \cite{anantharaman2023} that, as in the case of graphs, the functions
$V_g^\type(\ell)$ admit asymptotic expansions in powers of~$1/g$:
\begin{equation}
  \label{eq:vtype_exp}
  \frac{V_g^\type(\ell)}{V_g}
  = \sum_{k=\chi(\type)}^{\ord} \frac{f_k^\type(\ell)}{g^k} + \O[\epsilon,\ord,\chi(\type)]{\frac{e^{(1+\epsilon)\ell}}{g^{\ord+1}}}.
\end{equation}
The coefficients $(f_k^\type(\ell))_{k \geq 0}$ are the main characters of this story. 

\subsubsection{Friedman--Ramanujan functions}
\label{sec:friedm-raman-funct}

The second hero of our adventure is the notion of Friedman--Ramanujan function, which we introduced
in \cite{anantharaman2023}. 

\begin{defi}
  For $\rK, \rN$ two integers, we define the set $\cF^{\rK,\rN}$ as the set of locally integrable
  functions $f : \R_{\geq 0} \rightarrow \C$ such that there exists a polynomial function $p$ of degree
  $<\rK$ and a constant $c$ such that:
  \begin{equation}
    \label{eq:FR}
    \forall \ell \geq 0,
    \quad 
    |f(\ell) - p(\ell) e^\ell| \leq c (\ell+1)^{\rN-1} e^{\ell/2}.
  \end{equation}
  Such a function is called a \emph{Friedman--Ramanujan function} (of degree $(\rK,\rN)$).
  The set $\cF^{\rK,\rN}$ is equipped with the norm
  \begin{equation}
    \|f\|_{\cF^{\rK,\rN}}
    = \|p\|_{\infty} + \sup_{\ell \geq 0} \frac{|f(\ell) - p(\ell) e^\ell|}{(\ell+1)^{\rN-1} e^{\ell/2}},
  \end{equation}
  where $\|p\|_{\infty}$ is the $\ell^\infty$-norm on the space of polynomials. The set of
  \emph{Friedman--Ramanujan remainders} $\cR^{\rN} \subset \cF^{\rK,\rN}$ is defined as
  $\cF^{0,\rN}$, i.e. as the set of Friedman--Ramanujan functions with $p=0$.
\end{defi}

We similarly define weak spaces $\cF^{\rK,\rN}_w$ and $\cR^{\rN}_w$ equipped with a weak norm
$\|\cdot \|_{\cF^{\rK,\rN}}^w$ by replacing the definition \eqref{eq:FR} by an integrated version:
\begin{equation}
  \label{eq:FRw}
  \forall \ell \geq 0,
  \quad 
  \int_0^\ell|f(s) - p(s) e^s| \d s \leq c (\ell+1)^{\rN-1} e^{\ell/2}.
\end{equation}
As the name suggests, the strong definition implies the weak one.  Remark that in the work of
Friedman, where the functions are defined on $\Z_{> 0}$ instead of $\R_{\geq 0}$, this distinction
between weak and strong definition does not exist.

We are now ready to state the key argument we use to prove our spectral gap result.

\begin{thm}\label{thm:FR_type}
  Let ${\mathbf{T}}$ be a local type other than ``simple''. Then, for any $k \geq 0$, the function
  $\ell \mapsto f_k^\type(\ell)$ is a Friedman--Ramanujan function in the weak sense. Furthermore,
  its degree and norm can be bounded by a constant depending only on the order $k$ and
  $\chi(\type)$.
\end{thm}

The simple local type is singled out in this statement because its functions have a pole of order
$1$ at $0$ and therefore need to be multiplied by $\ell \mapsto \ell$.  Even though this is not the
method used in \cite{wu2022,lipnowski2024}, the probabilistic spectral gap
$\lambda_1 \geq \frac{3}{16}-\epsilon$ is related to the validity of \cref{thm:FR_type} for the
simple local type.  We prove \cref{thm:FR_type} in the case $\chi(\type) \in \{0, 1\}$ in
\cite{anantharaman2023}, and this is one of the main steps of our proof of the probabilistic bound
$\lambda_1 \geq \frac 29 - \epsilon$. The proof of \cref{thm:FR_type} in all generality is the bulk
of our new article \cite{anantharaman2025}.

\subsection{Stability of the Friedman--Ramanujan property by convolution}
\label{sec:stab-friedm-raman}

Let us conclude this section by presenting a fundamental idea on which our proof of
\cref{thm:FR_type} rests: the stability of the Friedman--Ramanujan hypothesis by convolution.

\subsubsection{The case of the usual convolution}
\label{sec:case-usual-conv}

Friedman mentions in \cite{friedman2003} that the class of $d$-Ramanujan functions is stable by
convolution. It is also the case of our Friedman--Ramanujan hypothesis, and we can make the
following elementary observation.

\begin{prp}
  \label{prp:stab_conv}
  If $f_i \in \cF^{\rK_i,\rN_i}$ for $i \in \{1,2\}$, then the convolution
  \begin{equation*}
    f_1 * f_2(\ell) = \int_0^\ell f_1(\ell_1) f_2(\ell-\ell_1) \d \ell_1
  \end{equation*}
  belongs in $\cF^{\rK,\rN}$ for $\rK=\rK_1+\rK_2$ and $\rN=\rN_1+\rN_2$.
\end{prp}

We can prove this by a straightforward (but not completely obvious) computation, examining each term
appearing as we replace $f_1$, $f_2$ by their approximations in the convolution. However, there is a
more elegant way to prove this stability, which we shall present as this is the way we proceed in
more complex cases.

The first observation is that \cref{prp:stab_conv} is a one-line upper bound in the case where both
functions are Friedman--Ramanujan remainders:
\begin{equation*}
  |f_1 * f_2(\ell)|
  \leq c_1 c_2  \int_0^\ell (\ell_1+1)^{\rN_1-1} (\ell-\ell_1+1)^{\rN_2-1} e^{\frac {\ell} 2} \d \ell_1
  \leq c_1c_2 (\ell+1)^{\rN-1} e^{\frac \ell 2}.
\end{equation*}
This is actually an important point to make: the hypothesis $\cR^{\rN}$ is quite a simple property,
as it is established by an upper bound. To contrast, the definition of $\cF^{\rK,\rN}$ is
substantially more challenging, as we wish to prove that the function has a very specific asymptotic
behavior, up to a high degree of precision.

The second key observation is the characterization of the space of Friedman--Ramanujan functions as
integral equations. It relies on the key characterization of the exponential function as the kernel
of the differential operator $\partial - \mathrm{Id}$, and more generally, functions of the form
$p(\ell)e^\ell$ for a polynomial $p$ as the kernel of its iterates.  We transform these differential
equations into integral equations to avoid dealing with regularity issue, defining two operators
acting on locally integrable functions, $\cP=\cP_x=\int_0^x$ (taking the primitive vanishing at $0$)
and $\cL=\cL_x=\mathrm{Id}-\cP_x$. It is trivial, nevertheless useful, to note that $\cP$ and $\cL$
preserve $\cF^{\rK,\rN}$ and $\cR^{\rN}$, as well as their weak versions. We have the following.

\begin{lem}\label{p:charFR} For $\rK$, $\rN$ two integers and $f : \R_{\geq 0} \rightarrow \R$
  locally integrable,
  \begin{equation*}
    f\in \cF^{\rK,\rN}
    \quad \Leftrightarrow \quad
    \cL^{\rK} f\in \cR^{\rN}.
    \end{equation*}
%    and same holds in the weak form.
  \end{lem}

  \begin{proof}
    We use standard ODE techniques and integral upper bounds, together with the fact that functions
    $p(\ell)e^\ell$ where $p$ is a polynomial of degree $< \rK$ form the kernel of the integral
    operator $\cL^{\rK}$.
  \end{proof}

  The third, and last, key ingredient is the following commutation relationship between the usual
  convolution and the integral operator $\cL$:  
  \begin{align}\label{e:easyconvo}
    \cL^{\rK_1+\rK_2} (f_1 * f_2)= \cL^{\rK_1} f_1 * \cL^{\rK_2} f_2.
  \end{align}
  We are now ready to prove the stability of the Friedman-Ramanujan hypothesis by convolution.

  \begin{proof}
    If $f_i\in \cF^{\rK_i,\rN_i}$, then $\cL^{\rK_i} f_i \in \cR^{\rN_i}$ by \cref{p:charFR}. Thus,
    using the stability in the case of Friedman--Ramanujan remainders, we obtain that
    \begin{equation}
      \cL^{\rK} (f_1 * f_2) = (\cL^{\rK_1}f_1) * (\cL^{\rK_2}f_2)\in \cR^{\rN},
    \end{equation}
    implying that $f_1 * f_2\in \cF^{\rK,\rN}$ by \cref{p:charFR} again, which is our claim.
  \end{proof}

  \subsubsection{But... Why are we talking about convolutions?}

  Let us give a vague and informal idea of the relevance of the notion of convolution to our
  problem. Recall that the functions we wish to study to prove a spectral gap result are average
  geometric counting functions, which count loops of a fixed topology on random objects. For
  instance, in the case of graphs, we will count all figure-eights of length $\ell$ on a random
  graph. It is then natural to see such a figure-eight as a concatenation of two simple loops, of
  respective lengths $\ell_1$ and $\ell_2$, such that $\ell_1+\ell_2=\ell$; see
  \cref{fig:conv_graph}. In a sense, this means that the counting function for the figure-eight will
  ``look like'' a sum of counting functions for pairs of simple loops of respective lengths $\ell_1$
  and $\ell - \ell_1$, where $\ell_1$ varies between $0$ and~$\ell$; in other words, a discrete
  convolution.

  \begin{figure}[hh]
    \centering
    \includegraphics{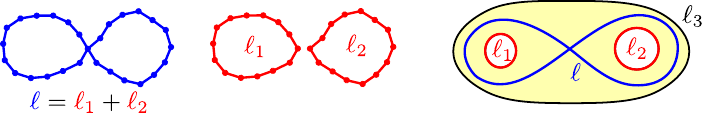} 
    \caption{Breaking the figure-eight in simpler loops, for graphs and hyperbolic surfaces.}
    \label{fig:conv_graph}
  \end{figure}

  As a consequence, in order to prove that the counting function for a figure-eight (or a more
  complicated topology) is a Ramanujan function, one could try and write it as a convolution of
  Ramanujan functions, and use the stability of this class. This becomes even more useful as we
  consider very intricate topologies, as required by the optimal spectral gap problem.

  In the case of hyperbolic surfaces, the situation is the same, but much more complex, and these
  differences are one of the core challenges which have no counterpart in the case of graphs. One
  can think that passing from discrete $\ell \in \Z_{>0}$ to continuous $\ell \in \R_{\geq 0}$ is a
  difficulty, which it is. However, the main difficulty is actually that we only count \emph{closed
    geodesics} rather than any closed loop. For instance, the figure-eight of length $\ell$ is not
  decomposed in two closed geodesics of lengths summing to $\ell$. One needs an additional geometric
  parameter to determine the geometry of \cref{fig:conv_graph} in the case of hyperbolic surfaces
  (e.g. the length $\ell_3$ of the third boundary component of the pair of pants). In this case, we
  have
  \begin{equation}
    \label{eq:ell_pop}
    \cosh \div{\ell} = 2 \cosh \div{\ell_1} \cosh \div{\ell_2} + \cosh \div{\ell_3},
  \end{equation}
  a nice enough formula, but arguably more complex than $\ell=\ell_1+\ell_2$.

  In \cite{anantharaman2023}, we proved that the volume function of the figure-eight can actually be
  viewed like a convolution of Friedman--Ramanujan functions, with an additional angle parameter
  (the angle at the intersection point of the eight). It is hard to see how this approach can be
  generalized to arbitrary complicated topologies. In \cref{sec:expr-length-funct}, we will provide
  general expressions for length functions which are similar to \eqref{eq:ell_pop}.  Importantly, we
  notice that if $\ell_1, \ell_2 \gg 1$ and $\ell_3$ is fixed, we have $\ell \approx \ell_1+\ell_2$
  in \eqref{eq:ell_pop}. Our expressions will satisfy similar properties, which will allow us
  compare our the volume functions with genuine convolutions of Friedman--Ramanujan functions, in a
  precise sense.
  
\subsubsection{Pseudo-convolutions and stability}
\label{sec:pseudo-conv-stab}

Due to the challenges brought to light in the previous paragraph, we need to define a notion of
pseudo-convolution, generalizing the usual notion of convolution. 

\begin{defi}
  We define the \emph{pseudo-convolution} of the functions $(f_i)_{1 \leq i \leq n}$ with level
  function $h$ and weight $\varphi$  by letting
   \begin{align*}
   f_1\star \ldots \star f_n|^h_\varphi(\ell) \d \ell
     = \int_{h(\x)=\ell}
     \varphi(\x) \prod_{i=1}^n f_i(x_i) \frac{\d \x}{\d \ell}
   \end{align*}
   where $\frac{\d \x}{\d \ell}$ denotes the volume induced by the Lebesgue measure on $\R^n$ on the
   level-set $\{\x \in \R_{\geq 0}^n \, : \, h(\x)=\ell\}$ of the height function $h$.
 \end{defi}
 This definition makes sense under reasonable assumptions on the functions involved. Notably, we
 want the level-sets of $h$ above to have a nice structure: we want that we can write the solution
 $x_1$ of the equation $h(x_1, \ldots, x_n)=\ell$ as a nice function of $\ell, x_2, \ldots, x_n$ on
 its domain of definition, and that we satisfy the hypotheses of the inverse function theorem so as
 to write $\frac{\d \x}{\d \ell} = \frac{\partial x_1}{\partial \ell} \d x_2\ldots \d x_n$ on the
 domain of the integral (and equivalent expressions for any choice of an expressed variable).

 Note that the case of the usual convolution corresponds to $h(\x)=x_1+\ldots+x_n$ and
 $\varphi \equiv 1$.  So, here, we allow for our convolution to no longer take place on a level-set
 $x_1 + \ldots + x_n=\ell$, but a more general kind of level-set, and we further introduce a
 weight. We will see such integrals naturally appearing in \cref{sec:writing-as-an}.

 This notion of pseudo-convolution is very wide and might not resemble an actual convolution at all,
 for general $h$ and $\varphi$. However, we will work under additional hypotheses on $h$ and
 $\varphi$, quantifying precisely how far away the $(h,\varphi)$-convolution is from the usual
 convolution. We define the following classes of functions.
  
 \begin{defi} \label{def:Ena} Let $a>0$.
   \begin{itemize}
   \item We denote as $E_n^{(a)}$ the space of $\mathcal{C}^\infty$ functions
     $\varphi : [a, +\infty)^n \rightarrow \R$ such that for every multi-index
     $\alpha\in \{0, 1\}^n$ with $\alpha\not=(0,\ldots, 0)$,
     \begin{equation}
       \sup_{x \in [a, +\infty)^n}\Big\{ e^{\alpha \cdot x}|\partial^\alpha \varphi(x)|\Big\} <+\infty.
     \end{equation}
   \item We let $\cE_n^{(a)}$ be the  space of functions $h : [a, + \infty)^n \rightarrow \R$
     such that the function $ (x_1, \ldots, x_n) \mapsto h(x_1, \ldots, x_n) - \sum_{j=1}^n x_j $
     belongs in $E_n^{(a)}$.
\end{itemize}
\end{defi}

We prove the following stability statement for the Friedman--Ramanujan hypothesis (under the natural
hypotheses on the functions $h$ and $\varphi$ allowing the $(h,\varphi)$-pseudo-convolution to be
well-defined).
 
\begin{thm} \label{t:intermediate} Let $a>0$. We assume that $h\in \cE_n^{(a)}$ and 
  $\varphi\in E_n^{(a)}$.  If $(f_j)_{1 \leq j \leq n}$ are continuous functions in $\cF$, then
  $ f_1\star \ldots \star f_n |^h_\varphi\in \cF$.
 \end{thm}

 Whilst the statement itself is not per se used in the proof of \cref{thm:dream}, its proof is
 extremely similar to the technique we use to establish \cref{thm:FR_type}. We therefore sketch it
 here.

 \begin{proof}
   The proof follows the same line as the proof for the usual convolution, albeit more elaborate due
   to the differences between the two operations and the more cumbersome definition here.
   More precisely, for each $1 \leq j \leq n$, we let $(\rK_j,\rN_j)$ be the Friedman--Ramanujan
   degree of $f_j$ so that $f_j \in \cF^{\rK_j,\rN_j}$. By \cref{p:charFR},
   $\cL^{\rK_j} f_j \in \cR^{\rN_j}$. For $\rK=\sum_{j=1}^n\rK_j$ and $\rN = \sum_{j=1}^n \rN_j$, we
   prove that the image of the pseudo-convolution $f_1 \ast \ldots \ast f_n|_\varphi^h$ by the
   operator $\cL^\rK$ belongs in $\cR^{\rN+n}$ which means, by \cref{p:charFR} again, that
   $f_1 \ast \ldots \ast f_n|_\varphi^h \in \cF^{\rK,\rN+n}$. For convenience, we assume without
   loss of generality that $\rK_j \leq \rN_j$ for each $1 \leq j \leq n$.

   Let us first, as in the case of the usual convolution, deal with the case where $\rK_j=0$ for all
   $j$, i.e. the case where all functions are Friedman--Ramanujan remainders. The proof is a simple
   upper bound.     Indeed, for $\ell \geq 0$, by definition,
   \begin{equation}
     \label{eq:upb_FR_conv}
      f_1 \ast \ldots \ast f_n|_\varphi^h(\ell)
     = \int \varphi(x_1, \ldots, x_n)
     \prod_{j=1}^n f_j(x_j) \;\frac{\partial x_1}{\partial \ell}
     \d x_2 \ldots \d x_n
  \end{equation}
  where the integral runs over the set of values $x_2, \ldots, x_n$ such that the equation
  $h(x_1, \ldots, x_n) = \ell$ has a solution $x_1 = x_1(\ell,x_2, \ldots, x_n)$.  Part of our global
  assumptions on $h$ and $\varphi$ are that $\big|\frac{\partial x_1}{\partial \ell}\big|$ and
  $|\varphi|$ are uniformly bounded.  We furthermore observe that, if $h \in \cE_n^{(a)}$, then
  there exists $l_0>0$ such that
  \begin{equation}
    \label{eq:CE}
     \sum_{j=1}^n x_j \leq h(x_1, \ldots, x_n) + l_0.
   \end{equation}
   We call such an inequality a comparison estimate.
   Applying the Friedman--Ramanujan hypothesis for each $f_j$, we obtain a bound of the form
   \begin{equation}
     \label{eq:appl_FR_0}
     \Big|\prod_{j=1}^n f_j(x_j)\big|
     \leq C \prod_{j=1}^n (x_j+1)^{\rN_j-1} e^{x_j/2}
   \end{equation}
   which implies, using the comparison estimate \eqref{eq:CE},
   \begin{equation}
     \Big|\prod_{j=1}^n f_j(x_j)\big|
     \leq C' (\ell+1)^{\rN-n} e^{\ell/2}
   \end{equation}
   on the level-set $h(x_1, \ldots, x_n)=\ell$, where $C'$ is a constant depending on $l_0$. Then,
   putting everything together, we can conclude that
   \begin{equation}
     \label{eq:upb_FR_conv2}
    \big| f_1 \ast \ldots \ast f_n|_\varphi^h(\ell)\big|
    \leq C'' (\ell+1)^{\rN-n} e^{\ell/2} \int_{\sum_{j=2}^n x_j \leq \ell + l_0} \d x_2 \ldots \d x_n,
  \end{equation}
  which leads to the claim since the integral above is bounded by $(\ell+l_0)^{n-1}$.
   
  Let us now proceed to the general case. In order to do so, we describe the way the operator $\cL$
  acts on pseudo-convolutions, and generalize \eqref{e:easyconvo} to the case of
  pseudo-convolutions.  The main thing is to make sure to break down the terms appearing so that we
  can prove the desired bound on each term. We prove
   \begin{equation}
     \cL^\rK(f_1 \ast \ldots \ast f_n|_\varphi^h)
     = \sum_{\pi, t, t'} 
     \cL^{t}\cP^{t'} (\partial^{\pi_1}\Phi_1 \ast \ldots \ast \partial^{\pi_n} \Phi_n
     |_{\partial^{\pi_0}\varphi}^{h})
   \end{equation}
   where the sum runs over a set of values $(\pi,t,t')$ as follows:
   \begin{itemize}
   \item the integers $t$ and $t'$ are so that $0 \leq t + t' \leq \rK$;
   \item $\pi = (\pi_0, \ldots, \pi_n)$ is a family of disjoint subsets of $\{1, \ldots, n\}$ such
     that $j \notin \pi_j$ for each $1 \leq j \leq n$ and, for all $j$,
     $\partial^{\pi_j} := \prod_{i \in \pi_j} \partial_{x_i}$;
   \item each function $\Phi_j$ is obtained from $f_j$ by one of the following transformations:
     \begin{itemize}
     \item we may have $\Phi_j = \cL_{x_j}^{\rK_j} f_j \times \partial_{x_j} h$;
     \item or $\Phi_j =  f_j \times (1 - \partial_{x_j} h)$;
     \item or $\Phi_j = \cP_{x_j}\big(\cL^t_{x_j} f_j \times (1-\partial_{x_j} h)\big) \times \partial_{x_j} h$ for an
       integer $0 \leq t < \rK_j$;
     \item or $\Phi_j = \cP_{x_j} \big( \cL^t_{x_j} f \times \partial_{x_j} h\big)$ for a
       $0 \leq t < \rK_j$, in which case $j \in \pi_0$.
     \end{itemize}
   \end{itemize}
   We then use the Friedman--Ramanujan hypothesis (which tells us that $\cL^{t}f_j$ grows at most
   like $(\ell+1)^{\rN_j-1} e^{x_j}$ if $t<\rK_j$ and that the exponential factor can be improved to be
   $e^{x_j/2}$ for $t=\rK_j$) as well as the hypotheses on $h$ and $\varphi$ to prove a bound of the
   form
   \begin{equation}
     \label{eq:appl_FR_1}
       \Big|\partial^{\pi_0} \varphi \prod_{j=1}^n \partial^{\pi_j} \Phi_j\Big|
       \leq C \prod_{j=1}^n (x_j+1)^{\rN_j} e^{x_j/2}.
     \end{equation}
     More precisely, we can bound the function $\partial^{\pi_j}\Phi_j$ individually by
     $(x_j+1)^{\rN_j}e^{x_j/2}$ in the first three cases above thanks to the hypotheses on $f_j$ and
     $h$ alone; in the last case, we use the hypothesis on $\varphi$ and the fact that $j \in \pi_0$
     to obtain a term $e^{-x_j/2}$ in the decay of $\partial^{\pi_0}\varphi$, balancing out the
     growth-rate $(\ell+1)^{\rN_j}e^{x_j}$ of $\partial^{\pi_j}\Phi_j$.  Equation
     \eqref{eq:appl_FR_1} then naturally replaces \eqref{eq:appl_FR_0} in the case $\rK_j=0$, and we
     conclude as before.
 \end{proof}
 
\section{Weil--Petersson volume functions}
\label{sec:volume}

For this second lecture, we shall now focus on the volume functions $V_g^{\mathrm{all}}(\ell)$ and
$V_g^\type(\ell)$ introduced in \cref{sec:volume-functions-1}.
% The type-specific volume function $V_g^\type(\ell)$ is defined by restricting the summation above
% to closed geodesics of a specific type $\type$, in a way we shall now precise.
We will provide a closed expression for the volume functions $V_g^{\type}(\ell)$ and full asymptotic
expansions for the ratio $V_g^\type(\ell)/V_g$ in powers of $1/g$.  Whilst our motivation in
studying these volume functions was initially to prove \cref{thm:dream} on the spectral gap of random
hyperbolic surfaces, the objects and results are purely geometric and we are hoping they will be
used to study other aspects of the geometry of random hyperbolic surfaces.

\subsection{Splitting the average and local topological type}
\label{sec:defin-volume-funct}

% \subsubsection{The contribution of simple geodesics}
% \label{sec:case-simple-geod}

We recall that the overall counting function $V_g^{\mathrm{all}}(\ell)$ is determined by the
property
\begin{equation}
  \label{e:def_av_all}
  \av{F} := \Ewp{\sum_{\gamma \in \cG(X)} F(\ell_X(\gamma))}
  = \frac{1}{V_g} \int_0^{+ \infty} F(\ell) V_g^{\mathrm{all}}(\ell) \d \ell
\end{equation}
where $\cG(X)$ is the set of primitive closed geodesics on $X$, and for any $\gamma \in \cG(X)$,
$\ell_X(\gamma)$ denotes the length of $\gamma$.  In order to understand and study the average
$\av{F}$ and associated volume $V_g^{\mathrm{all}}(\ell)$, it is natural to split the sum depending
on the topology of the closed geodesics in it:
\begin{equation}
  \label{eq:def_av_T}
  \av[\mathrm{all}]{F} = \sum_{\type} \av[\type]{F}
  \quad \text{with} \quad
  \av[\type]{F}
  := \Ewp{\sum_{\gamma \sim \type}F(\ell(\gamma))}
\end{equation}
where we write $\gamma \sim \type$ to mean that $\gamma$ ``has the topology~$\type$''.  This
means we rewrite the overall volume function as
\begin{equation}
  \label{eq:vol_T}
  V_g^{\mathrm{all}}(\ell) = \sum_{\type}  V_g^{\type}(\ell)
\end{equation}
where $V_g^{\type} : \R_{>0} \rightarrow \R_{\geq 0}$ is the volume function defined by the property
that, for any test function $F$,
\begin{equation}
  \label{eq:av_T}
  \av[\type]{F}
    = \frac{1}{V_g} \int_0^{+ \infty} F(\ell) V_g^{\type}(\ell) \d \ell.
\end{equation}

\subsubsection{Simple v.s. non-simple}
\label{sec:simple_vs}

There is one specific kind of volume function that has been very well-studied in the literature: the
function $V_g^{\mathbf{s}}(\ell)$ defined by
\begin{equation}
  \label{eq:av_S}
  \av[\mathbf{s}]{F}
  = \Ewp{\sum_{\gamma \text{ simple}}F(\ell(\gamma))}
  = \frac{1}{V_g} \int_0^{+ \infty} F(\ell) V_g^{\mathbf{s}}(\ell) \d \ell,
\end{equation}
i.e. the contribution of simple closed geodesics to the overall average $\av{F}$. Mirzakhani
provided in \cite{mirzakhani2007} a simple expression for $V_g^{\mathbf{s}}(\ell)$ in terms of
volumes of moduli spaces, which we will present in \cref{sec:simple-local-type}.

In the regime $g \rightarrow + \infty$, one can prove that the volume function
$V_g^{\mathrm{all}}(\ell)$ is well-approximated by the contribution $V_g^{\mathbf{s}}(\ell)$ of
simple closed geodesics. This observation was first made by Mirzakhani and Petri in
\cite{mirzakhani2019}, and then vastly improved in \cite{wu2022}, where Wu and Xue proved that
\begin{equation}
  \label{eq:approx_s}
  V_g^{\mathrm{all}}(\ell)
  = V_g^{\mathbf{s}}(\ell) 
  + \mathcal{O}_\epsilon \Bigg(\frac{e^{(1+\epsilon) \ell}}{g} \Bigg).
\end{equation}
We can view this as an example of splitting by type: the topologies being ``simple''
v.s. ``non-simple''.
This is a very coarse splitting and lacks precision. Indeed, we will see in the following that the
term of size $1/g^K$ in the asymptotic expansion of $V_g^{\mathrm{all}}(\ell)/V_g$ contains
contributions of more and more intricate non-simple geodesics as the precision $K$ increases, which
we will need to  describe individually.

\subsubsection{$\MCG$-types}

Another approach present in the literature consists in splitting the sum $\av{F}$ according to
\emph{$\MCG$-types} of closed geodesics.

The mapping-class-group $\MCG(X)$ of $X$ is the set of positive diffeomorphisms of
$X$, up to isotopy. This group naturally acts on loops on the surface $X$, and we say two loops
$\gamma_1$ and $\gamma_2$ on $X$ have the same $\MCG$-type if there exists an element of $\MCG(X)$
sending $\gamma_1$ to a loop homotopic to $\gamma_2$.

Splitting the sum \eqref{e:def_av_all} by $\MCG$-type is the finest way we could split it in a way
which then allows to make sense of the averages $\av[\type]{F}$. This approach is prevalent in
Mirzakhani's work, where these types are called \emph{topological types}.

Simple closed geodesics are then split in $g$ different topological types: those which are
non-separating, and those which separate the surface into two connected components, of genus $1 \leq
i \leq g-1$ on the left of the loop, and $g-i$ on its right. See \cref{fig:simp}.
\begin{figure}[h]
  \centering
  \includegraphics{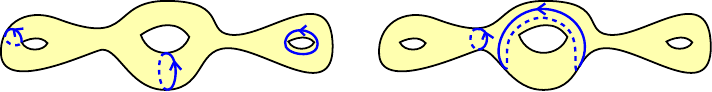}
  \caption{Examples of simple loops on a surface of genus $3$: non-separating on the left, and
    separating into a surfaces of genus $1$ and $2$ on the right.}
  \label{fig:simp}
\end{figure}

\subsubsection{Local types}
\label{sec:local-type-loops}

In our work, we introduce a new way to split the sum, by \emph{local topological type}. This
splitting is coarser than Mirzakhani's, but much more precise than the simple v.s. non-simple
splitting. It lies at a soft spot where it is precise enough to be able to write simple expressions
for each term $V_g^\type(\ell)$, but also groups a lot of similar contributions together, reducing
the need for tedious topological enumerations.

Our notion of local type consists in zooming in and looking at the closed geodesic~$\gamma$ on $X$,
and its neighborhood, but forgetting what happens outside of this neighborhood. So, for instance,
all simple closed geodesics belong to the same local type because, locally, they all ``look like''
the circle $\{0\} \times \R \diagup \Z$ inside the cylinder $[-1,1] \times \R \diagup \Z$.

We would naturally like to define the neighborhood of the geodesic $\gamma$ by taking its regular
neighborhood $\cN$, and then considering the pair $(\cN,\gamma)$, as done above in the case of
simple closed geodesics.  However, the pair $(\mathcal{N}, \gamma)$ is not defined as a function of
the homotopy class of $\gamma$ solely. This is an issue since, as we vary the metric, the geodesic
representatives of one homotopy class of loops for two distinct metrics might not be isotopic. An
example of this phenomenon is represented in \cref{fig:reid}.

\begin{figure}[h]
  \centering
  \includegraphics{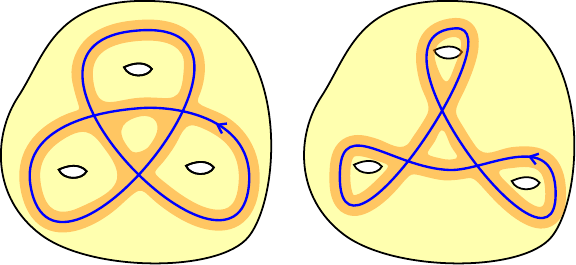}
  \caption{The two blue loops are homotopic but not isotopic.}
  \label{fig:reid}
\end{figure}

Fortunately, we know by \cite{graaf1997} that any two homotopic loops in minimal position differ
only from a finite number of third Reidemeister moves together with an isotopy. Such a move is
exactly the difference between the two loops in \cref{fig:reid}. It is associated with contractible
components in the complement of $\cN$ in $X$. This motivates the following definition.

\begin{defi}
  If $\gamma$ is a loop on $X$, we define the \emph{surface $S(\gamma)$ filled by $\gamma$} by gluing
  to its regular neighborhood $\mathcal{N}$ any contractible connected component
  of~$X \setminus \mathcal{N}$.
\end{defi}

The surface $S(\gamma)$ is then filled by $\gamma$, i.e., all of the connected components of
$S(\gamma) \setminus \gamma$ are contractible. The pair $(S(\gamma), \gamma)$ is well-defined, up to
isotopy, as a function of the homotopy class of $\gamma$.

\begin{expl}
  The surface filled by a non-contractible simple loop is a cylinder. A figure-eight fills a pair of
  pants, as represented in \cref{fig:conv_graph}. In \cref{fig:reid}, the filled surface is
  constructed by adding the central disk to the regular neighborhood, hence obtaining a surface of
  genus $0$ with $4$ boundary components. Notice how the two loops are homotopic in $S(\gamma)$ but
  not in $\cN$.
\end{expl}

We are now ready to define our notion of local type.

\begin{defi}
  Let $X$ be a compact hyperbolic surface. Two loops $\gamma_1$, $\gamma_2$ on $X$ are said to
  \emph{belong to the same local type} if there exists an positive homeomorphism
  $S(\gamma_1) \rightarrow S(\gamma_2)$ sending $\gamma_1$ on a loop homotopic to $\gamma_2$.
\end{defi}

Constructed this way, the notion of local type encompasses the local topology of the geodesic itself
but not the way it is embedded in the genus $g$ surface.  The notion of local type is independent of
the choice of a representative in the homotopy class. Two loops which have the same $\MCG$-type also
have the same local type, so this notion is coarser.

\begin{expl}
  If $\gamma_1$, $\gamma_2$ are two non-contractible simple loops on the surface $X$, then they each
  fill a cylinder and it is easy to define a homeomorphism sending $S(\gamma_1)$ on $S(\gamma_2)$
  and $\gamma_1$ on $\gamma_2$. Hence, all simple closed geodesics are grouped in one local type,
  the \emph{simple local type}, associated to the average $\av[\mathbf{s}]{F}$ and the volume
  $V_g^{\mathbf{s}}(\ell)$ introduced in \cref{sec:simple_vs}.
\end{expl}

As illustrated by the case of simple loops, the information we ``forget'' when going from
$\MCG$-type to local type is the topology of the complement $X \setminus S(\gamma)$.  We can
therefore define the notion of local type intrinsically, in the following way.

  \begin{nota}\label{n:FT}
    To any pair of integers $(g_\Sf, n_\Sf)$ such that $ 2g_\Sf-2+n_\Sf \geq 0$, we shall associate a
    \emph{fixed} smooth oriented surface $\Sf$ of signature $(g_\Sf,n_\Sf)$. We denote by
    $\chi(\Sf)= 2g_\Sf-2+n_\Sf$ the absolute Euler characteristic of $\Sf$.  The data of the pair of
    integers $(g_\Sf, n_\Sf)$, or equivalently of the surface $\Sf$, is called a \emph{filling type}.
  \end{nota}

  \begin{defi}
  \label{def:local_type}
  A \emph{local loop} is a pair $(\Sf,\curve)$, where $\Sf$ is a filling type and $\curve$ is a
  primitive loop filling $\Sf$. Two local loops $(\Sf,\curve)$ and $(\Sf',\curve')$ are said to be
  \emph{locally equivalent} if the following two conditions are satisfied:
  \begin{itemize}
  \item $\Sf=\Sf'$ (i.e.  $g_\Sf = g_{\Sf'}$ and $n_\Sf=n_{\Sf'}$);
  \item there exists a positive homeomorphism $\psi : \Sf \rightarrow \Sf$, possibly
    permuting the boundary components of $\Sf$, such that $\psi \circ \curve$ is homotopic to~$\curve'$.
  \end{itemize}
This defines an equivalence relation on local loops. Equivalence classes for this relation are
denoted as $\type=\eqc{\Sf, \curve}$ and called \emph{local (topological) types} of loops. The absolute
Euler characteristic of a local type is defined by $\chi(\type):=\chi(\Sf)$.
\end{defi}

\begin{defi}
  Let $X$ be a compact hyperbolic surface. For $\gamma \in \mathcal{G}(X)$, we say $\gamma$
  \emph{has local topological type $\type$}, and write $\gamma \sim \type$, if the local loop
  $(S(\gamma), \gamma)$ is a representative of the equivalence class $\type$.
\end{defi}

\begin{figure}[h]
  \centering
  \includegraphics{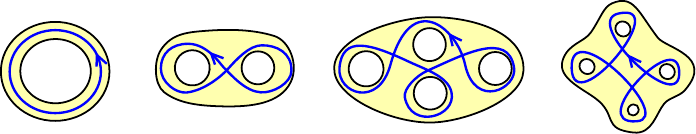}
  \caption{Examples of local topological types.}
  \label{fig:examples_types}
\end{figure}

Examples of local types are represented in Figure \ref{fig:examples_types}. The left-most type is
the simple local type, which fills a cylinder. The second is the figure-eight, which fills a pair of
pants. The last two examples are of filling type $(0,4)$: the first is an example of the class of
\emph{generalized eights} we shall be interested in moving forward.

\subsection{Writing as an integral}
\label{sec:writing-as-an}

Now that we have defined the notion of local type $\type$, and hence the volume function
$V_g^\type(\ell)$, we shall provide a method to compute it. The first step is to write it as an
integral as we have done in \cite{anantharaman2023}.

\subsubsection{Bordered hyperbolic surfaces and Teichm\"uller spaces}
\label{sec:bord-hyperb-surf}

In the following, because we will cut our compact surface along some geodesics, we will need to
consider hyperbolic surfaces with boundary. For $(g,n)$ two integers with $2g-2+n>0$ and
$\x=(x_1, \ldots, x_n) \in \R_{>0}^n$, we let $\cM_{g,n}(\x)$ denote the \emph{moduli space} of
hyperbolic surfaces of genus $g$ with $n$ labelled geodesic boundary components of respective
lengths $x_1, \ldots, x_n$, up to isometry preserving the labelling of the boundary components. 

The universal cover of $\cM_{g,n}(\x)$ is the Teichm\"uller space $\cT_{g,n}(\x)$ and can be
realized the following way. An element of $\cT_{g,n}(\x)$ is the data of a pair $(X,\phi)$, where
$X$ is a hyperbolic surface as above, and $\phi : S_{g,n} \rightarrow X$ is a \emph{marking}, i.e. a
positive diffeomorphism from a fixed base surface $S_{g,n}$ of signature $(g,n)$ to~$X$. We identify
a pair $(X_1,\phi_1)$, $(X_2,\phi_2)$ in $\cT_{g,n}(\x)$ if and only if there exists an isometry
$i:X_1 \rightarrow X_2$ such that $i \circ \phi_1$ and $\phi_2$ are isotopic.

The mapping-class-group $\MCG_{g,n}$ of $S_{g,n}$ acts on $\cT_{g,n}(\x)$ by precomposition of the
marking: $\psi \cdot (X,\phi) := (X,\phi \circ \psi^{-1})$. The moduli space $\cM_{g,n}(\x)$ then is
the quotient of $\cT_{g,n}(\x)$ for this action.

The space $\cT_{g,n}(\x)$ is equipped with the Weil--Petersson symplectic form, which induces the
Weil--Petersson volume form $\Volwp[g,n,\x]$. This volume form is invariant by action of the
mapping-class-group and hence descends to a volume form on the moduli space $\cM_{g,n}(\x)$, denoted
the same way. Whilst the volume of $\cT_{g,n}(\x)$ is infinite, that of $\cM_{g,n}(\x)$ is
finite, and we shall denote it as $V_{g,n}(\x)$. Mirzakhani proved in \cite{mirzakhani2007} that
$V_{g,n}(\x)$ is a polynomial function in $\x$ and provided recursive formulas allowing to compute
it.

\subsubsection{The case of the simple local type}
\label{sec:simple-local-type}

In a breakthrough article \cite{mirzakhani2007}, Mirzakhani proved that one can compute the volume
function $V_g^{\mathbf{s}}(\ell)$ by a beautiful argument. If a hyperbolic surface~$X$ of genus~$g$
contains a simple closed geodesic of length $\ell$, then, cutting $X$ along this closed geodesic, we
obtain a new hyperbolic surface, with two geodesic boundary components of length exactly $\ell$ (the
resulting surface might be connected or not, depending on the $\MCG$-type of $\gamma$). To recover
the initial surface, we should glue this surface back along the geodesic, with a degree of freedom,
the twist parameter, an element of $\R \diagup \ell \Z$.

\begin{figure}[h]
  \centering
  \includegraphics{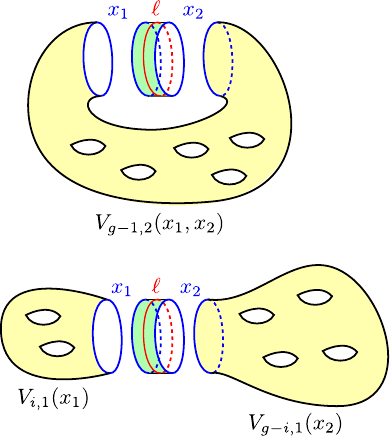}
  \caption{Illustration of the formula in the case of simple geodesics. }
  \label{fig:simple}
\end{figure}

It follows that the set of hyperbolic surfaces of genus $g$ containing a simple closed geodesic of
length exactly $\ell$ can be identified to a union of moduli spaces of bordered hyperbolic surfaces
with two boundary components of length $\ell$, with addition of a twist parameter in
$\R \diagup \ell \Z$.  This argument can be made rigorous and yields the beautiful closed
expression
\begin{equation}
  \label{eq:Vsimple}
  V_g^{\mathbf{s}}(\ell) = \ell V_{g-1,2}(\ell,\ell) + \ell \sum_{i=1}^{g-1} V_{i,1}(\ell) V_{g-i,1}(\ell).
\end{equation}
The sum is an enumeration of all possible topologies for the surface obtained once we cut our genus
$g$ surface along our simple closed geodesic, or, alternatively, all $\MCG$-types in our local type.
We represent the enumeration in Figure \ref{fig:simple}.
Note that, actually, Mirzakhani proved this expression for the volume functions associated to
$\MCG$-types, and \cref{eq:Vsimple} is obtained by summing all of these volume functions.

I promised you a formula in an integral form. Let's do this in a bit of a silly manner -- which will
make a lot more sense for general local types.
The surface filled by a simple closed geodesic is a cylinder. The space of Riemannian metrics on the
cylinder is naturally identified to $\R_{>0}$. If we were to consider the two boundary components of
the cylinder as free variables $x_1$, $x_2$, then we can rather view this space as $\R_{>0}^2$
equipped with the measure 
\begin{equation*}
  \d \Volwp[0,2](x_1,x_2) := V_{0,2}(x_1,x_2) \d x_1 \d x_2
  \quad \text{where} \quad
  V_{0,2}(x_1,x_2) = \frac{1}{x_1} \delta(x_1-x_2)
\end{equation*}
and $\delta$ is the Dirac delta function\footnote{The division by $x_1$ in the definition of
  $V_{0,2}(x_1,x_2)$ corresponds to the fact that we have one too many twist parameters when we glue
  two boundary components to either side of a cylinder.}. We can then rewrite
\begin{equation}
  \label{eq:Vsimpleint}
  V_g^{\mathbf{s}}(\ell)
  = \int_{\substack{x_1=x_2=\ell}}
  \Phi_{g}^{(0,2)}(x_1,x_2) \frac{\d \Volwp[0,2](x_1,x_2)}{\d \ell}
\end{equation}
where:
\begin{itemize}
\item the integral runs over the set of $(x_1, x_2) \in \R_{>0}^2$ such that $x_1=x_2=\ell$,
and the measure $\d \Volwp[0,2](x_1,x_2)/\d \ell$ is the measure induced by $\d
  \Volwp[0,2](x_1,x_2)$ on this level-set;
\item the function
  \begin{equation}
    \label{eq:psisimple}
    \Phi_{g}^{(0,2)}(x_1,x_2) 
    := x_1 x_2 V_{g-1,2}(x_1,x_2) + x_1 x_2\sum_{i=1}^{g-1} V_{i,1}(x_1) V_{g-i,1}(x_2)
  \end{equation}
  counts the possible metrics for the complement of the cylinder of boundary components of lengths
  $x_1$ and $x_2$ in the surface of genus $g$.
\end{itemize}

\subsubsection{General local types}
\label{sec:general-local-types}

We prove an integral expression of the same form for the volume function $V_g^\type(\ell)$ for
general local types $\type = \eqc{\Sf,\curve}$ in \cite[\S 5]{anantharaman2023}.

\paragraph{Domain of the integral}

In this case, the integral will run on the set of marked hyperbolic metrics on the filled surface
$\Sf$, the space
\begin{equation}
  \label{eq:def_teich}
  \cT^*_{\g,\n} = \{(\x, Y), \x \in \R_{>0}^{\n}, Y \in \cT_{\g,\n}(\x)\}.
\end{equation}
This space is naturally equipped with the measure
\begin{equation}
  \label{eq:vol_teich}
  \d \Volwp[\g,\n](\x,Y) = \d \x \d \Volwp[\g,\n,\x](Y)
\end{equation}
where $\d \x$ denotes the Lebesgue measure on $\R_{>0}^{\n}$. We will sometimes omit the mention of
$\x$, and write $Y \in \cT_{\g,\n}^*$ when we do not care to draw particular attention to the
boundary lengths $\x$ of $Y$.

Because $Y \in \cT_{\g,\n}^*$ is a marked surface, we can view any homotopy class of loops $\curve$
on $\Sf$ as a homotopy class of loops on $Y$. We then define the length-function
$\cT_{\g,\n}^* \ni Y \mapsto \ell_Y(\mathbf{c})$ by taking the length of the geodesic representative
in the homotopy class of  $\mathbf{c}$ for the hyperbolic metric on $Y$.

This allows to define the level-sets
\begin{equation}
  \label{eq:level-set}
  \{Y \in \cT_{\g,\n}^* \, : \, \ell_Y(\curve) = \ell\}
\end{equation}
for $\ell \in \R_{>0}$. The Weil--Petersson volume $\d \Volwp[\g,\n](\x,Y)$ naturally induces a volume on
these level-sets, denoted by $\d \Volwp[\g,\n](\x,Y)/\d \ell$, and defined by the property that, for
any integrable function $F : \cT_{\g,\n}^* \rightarrow \R$,
\begin{equation*}
  \int_{\cT_{\g,\n}^*} F(\x,Y) \d  \Volwp[\g,\n](\x,Y)
  = \int_0^{+ \infty} \Big(\int_{\ell_Y(\curve)=\ell} F(\x,Y) \frac{\d \Volwp[\g,\n](\x,Y)}{\d \ell} \Big)\d \ell.
\end{equation*}
This idea of decomposing integrals on level-sets should be reminiscent of our notion of
pseudo-convolution introduced in \cref{sec:pseudo-conv-stab}. We shall indeed later view the
integrals we will find as pseudo-convolution on the space $\cT_{\g,\n}^*$ with the height function
$h_\curve: Y \mapsto \ell_Y(\curve)$.

\paragraph{Topological enumeration}

As in the case of simple loops, we shall enumerate all ways to embed the filled surface $\Sf$ in a
surface of genus $g$. We do so by defining a notion of \emph{realization} $\mathfrak{R}$ in \cite[\S
4]{anantharaman2023}, which is the combinatorial data of the topology of the complement of $\Sf$ in
the surface of genus $g$. More precisely, a realization $\mathfrak{R}$ is given by an integer
$\mathfrak{q} \geq 1$, a partition $\vec{I}=(I_1, \ldots, I_{\mathfrak{q}})$ of $\{1, \ldots, \n\}$
and integers $(g_1, \ldots,g_{\mathfrak{q}})$ such that, if $n_i := \#I_i$, then $2g_i-2+n_i>0$ or
$(g_i,n_i)=(0,2)$.  By adding the additional constraint that the numbering of the partition
$\vec{I}$ is fixed so that $i \mapsto \min(I_i)$ is increasing, we obtain an exact enumeration of
all embedding of $\Sf$ in a surface of genus $g$. We denote as $R_g(\Sf)$ the set of such
realizations.

To a realization we can associate a volume function
\begin{equation}
  \label{eq:volR}
  V_{\mathfrak{R}}(x_1, \ldots, x_{\n})
  = \prod_{i=1}^{\mathfrak{q}} V_{g_i,n_i}(x_j, j \in I_i)
\end{equation}
with the convention that $V_{0,2}(x,y)=\frac{1}{x}\delta(x-y)$.
We then set 
\begin{equation}
  \label{eq:def_phi_S}
  \Phi_g^\Sf(x_1, \ldots, x_{\n}) := x_1 \ldots x_{\n}
  \sum_{\mathfrak{R} \in R_g(\Sf)}
  V_{\mathfrak{R}}(x_1, \ldots, x_{\n}).
\end{equation}
The distribution $\Phi_g^{\Sf}$ counts the possible metrics for the complement of $\Sf$ in a genus
$g$ surface. It is a straightforward generalization of the formula \eqref{eq:psisimple} for the
cylinder, i.e. the filling type $(0,2)$.  Note that $\Phi_g^\Sf(\x)$ is a distribution rather than a
function due to the presence of the Dirac delta function, associated to the possible cylinders in
the complement of~$\Sf$.

\begin{expl}
  Let us explicit the distribution $\Phi_g^{\mathbf{P}}(x_1,x_2,x_3)$ for the pair of pants
  $\mathbf{P}$, i.e. the filling type $(0,3)$. This requires to enumerate all embeddings of a pair
  of pants $\mathbf{P}$ in a surface of genus $g$. We find that
  \begin{equation}
    \begin{split}
      \Phi_g^{\mathbf{P}}(x_1,x_2,x_3)
      = & \, x_1 x_2 x_3
       V_{g-2,3}(x_1,x_2,x_3) \\
        &+ x_1 x_2 x_3 \sum_{\substack{\{i_1, i_2, i_3\} \\ = \{1,2,3\}}}
      \sum_{i=0}^{g-2}   V_{i,2}(x_{i_1}, x_{i_2}) V_{g-i-1,1}(x_{i_3}) \\
        &+ x_1 x_2 x_3 \sum_{g_1+g_2+g_3=g} V_{g_1,1}(x_1) V_{g_2,1}(x_2) V_{g_3,1}(x_3)
          \Big]
    \end{split}
  \end{equation}
   as illustrated in Figure \ref{fig:pop}. Note that the
   terms with $i=0$ in the second sum  correspond to situations where we glue a cylinder to two
   components of $\mathbf{P}$.
   \label{exa:int_pop}
\end{expl}

\paragraph{Our formula}

We are now ready to state our formula.

\begin{thm}
  \label{thm:express_average}
  Let $\type = \eqc{\Sf,\curve}$ be a local topological type. For any $g \geq 3$, 
  \begin{equation}
    \label{eq:orbit_decomposition}
    V_g^\type(\ell)
    = \frac{1}{n(\type)}
    \int_{\ell_Y(\curve)=\ell} \Phi_g^\Sf(\x) \, \frac{\d \Volwp[\g,\n](\x,Y)}{\d \ell}
  \end{equation}
  where the integral runs over the level-set $\{ Y \in \cT_{\g,\n}^*: \ell_Y(\curve)=\ell\}$, the
  distribution~$\Phi_{g}^\Sf$ is defined in \eqref{eq:def_phi_S}, and the integer coefficient
  $n(\type)$ is the cardinal of the group of positive homeomorphisms of $\Sf$ stabilizing the
  homotopy class of $\curve$, up to isotopy.
\end{thm}

Let us explicit the formula in a non-simple elementary example: local types of filling type $(0,3)$,
i.e. loops filling a pair of pants.

\begin{expl}
  A key property of pairs of pants is the their metric is entirely determined by the length of their
  boundary components and  $\cT_{0,3}^* = \R_{>0}^3$.
    \begin{figure}[h!]
  \centering
  \includegraphics{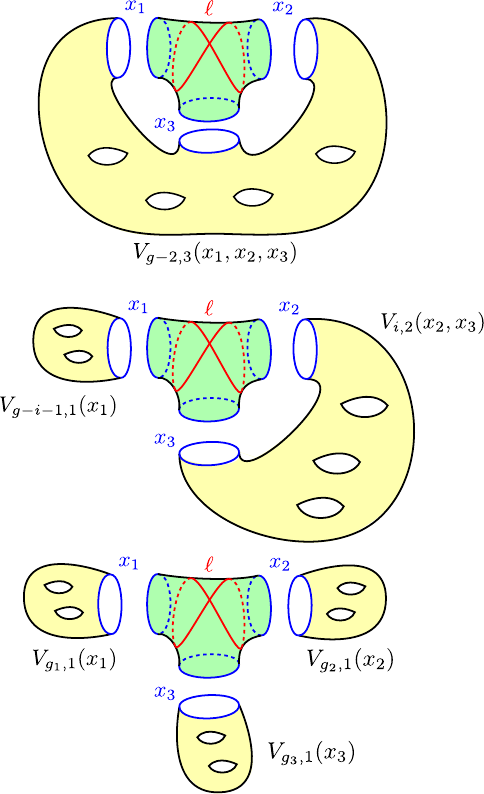}
  \caption{Illustration of the formula in the case of the figure-eight.}
  \label{fig:pop}
\end{figure}
Let $\curve$ be a loop filling the pair of pants $\mathbf{P}$. The length of $\curve$ is an analytic
function of $\x = (x_1,x_2,x_3) \in \cT_{0,3}^*$.  If for instance $\curve$ is a figure-eight, then
its length satisfies
  \begin{equation}
    \label{eq:length_eight}
    \cosh \div{\ell} = 2 \cosh \div{x_1} \cosh \div{x_2} + \cosh \div{x_3}
  \end{equation}
  so that we can express $x_3$ as a function of $\ell, x_1, x_2$:
  \begin{equation}
    x_3(\ell,x_1,x_2) = 2 \argch \Big( \cosh \div{\ell} - 2 \cosh \div{x_1} \cosh \div{x_2} \Big)
  \end{equation}
  which is defined as soon as
  $\cosh \div{x_1} \cosh \div{x_2}\leq M(\ell) := \frac{1}{2}(\cosh \div{\ell}-1)$.  Our formula
  then reads
  \begin{equation*}
    V_g^{\eqc{P, \curve}}(\ell)
    = \frac{\sinh \div{\ell}}{n(\curve)}
    \iint_{\cosh \div{x_1} \cosh \div{x_2} \leq M(\ell)}
    \frac{\Phi_g^{\mathbf{P}}(x_1, x_2, x_3(\ell,x_1,x_2))}{\sinh \div{x_3(\ell,x_1,x_2)}}
    \d x_1 \d x_2
  \end{equation*}
  because the measure induced by the Lebesgue measure on the level-sets is
  \begin{equation}
    \frac{\d x_1 \d x_2 \d x_3}{\d \ell}
    = \frac{\partial x_3}{\partial \ell} \d x_1 \d x_2
    = \frac{\sinh \div \ell}{\sinh \div{x_3}} \d x_1 \d x_2.
  \end{equation}
  We can perform the same operations fixing $x_2$ and $x_3$, and viewing $x_1$ as a function of
  $\ell, x_2, x_3$, and this yields different expressions for the same quantity. 
\end{expl}

A great feature of the formula \eqref{eq:orbit_decomposition} for $V_g^\type(\ell)$ is that its
dependency on the ambient genus $g$ is only through the function $\Phi_g^\Sf(\x)$, enumerating all
embeddings of $\Sf$ in a genus $g$ surface. In a way, the formula cuts the analysis of the volume
function $V_g^\type(\ell)$ into two completely separate questions.
\begin{itemize}
\item On the one hand, we need to understand the large-genus limit, which is entirely captured in
  the function $\Phi_g^\Sf(\x)$ (which only depends on the filled surface $\Sf$, and not the loop
  $\curve$).
\item On the other hand, we need to understand the level-sets of the function $Y \mapsto
  \ell_Y(\curve)$ on the Teichm\"uller space $\cT_{\g,\n}^*$. This is intimately connected to the
  topology of the loop $\curve$.
\end{itemize}

In the case of the simple local type, or local types filling a pair of pants, the expression of
$\d \Volwp[\g,\n](\x,Y)$ is very simple because the metric on $\Sf$ is entirely defined by the
length $\x$ of the boundary of $\Sf$. In all other cases, this is not the case, and describing the
measure $\d \Volwp[\g,\n](\x,Y)/\d \ell$ is one of the key challenges we tackle in this project.

\subsection{Asymptotic expansion in powers of $1/g$}
\label{sec:asympt-expans-powers}

Let us address the first issue raised above: understanding the large-genus limit. We provided in
\cite{anantharaman2022} asymptotic expansions for the volume polynomials $V_{g,n}(\x)$ in the
large-genus limit. Such expansions had already been considered by Mirzakhani and Zograf in
\cite{mirzakhani2015}, but the novelty in our article is that we focus on the behavior of the
asymptotic expansion as a function of the length variables $\x \in \R_{>0}^{\n}$, which is what we care
about here.

By standard techniques for asymptotic expansions, we deduce from the expression of $\Phi_g^\Sf(\x)$
together with the expansions for the volume polynomials $V_{g,n}(\x)$ that we can write an expansion
of the form
  \begin{equation}
    \label{e:claim_exp_phi_T}
    \frac{\Phi_g^\Sf(\x)}{V_g} = \sum_{k=\chi(\Sf)}^\ord \frac{\psi_k^\Sf(\x)}{g^k}
    + \mathcal{O}_{\ord,\chi(\Sf)}
    \left(\frac{(\|\x\|+1)^{c_\ord^\Sf}}{g^{\ord+1}} \exp{\div{x_1+\ldots+x_{\n}}}\right).
  \end{equation}
Most importantly, we obtain that the terms of this expansion can be written as linear combinations
of the distributions
\begin{equation}
  \label{eq:form_coeff_phi}
  \prod_{i \in V} x_i^{k_i}
  \prod_{i \in V_+} \cosh \div{x_i}
  \prod_{i \in V_-} \sinh \div{x_i}
  \prod_{j=1}^k x_{i_j} \delta(x_{i_j}-x_{i_j'}),
  \end{equation}
  where $V_0, V_+, V_-$ are disjoint subsets of $\{1, \ldots, \n\}$, $V:=V_0 \cup V_+ \cup V_-$,
  $(k_i)_{i \in V}$ are integers, and $\{(i_j,i_j')\}_{1 \leq j \leq k}$ is a perfect matching of
  the complement of $V$.
  
  As a consequence, now using \cref{thm:express_average}, we obtain that the volume functions
  $V_g^\type(\ell)/V_g$ admit asymptotic expansions in powers of $1/g$, and that the terms of this
  expansion can be expressed as a linear combination of functions of the form
  \begin{equation}
    \label{eq:form_exp}
    \int_{\ell_Y(\curve)=\ell}
    \prod_{i \in V} x_i^{k_i}
    \prod_{\epsilon = \pm}
      \prod_{i \in V_\epsilon} \hyp_\epsilon \div{x_i}
      \prod_{j=1}^k x_{i_j} \delta(x_{i_j}-x_{i_j'})
      \frac{\d \Volwp[\g,\n](\x,Y)}{\d \ell}
  \end{equation}
  where we introduce the shorthand notation
  \begin{equation}
    \label{eq:hyp}
    \hyp_\epsilon(x) :=
    \frac{\exp(x) + \epsilon \exp(-x)}{2}
    =
    \begin{cases}
      \cosh(x) & \text{ if } \epsilon = + \\
      \sinh(x) & \text{ if } \epsilon = -.
    \end{cases}
  \end{equation}
  From there onward, in the rest of this section, there is no more mention of the ambient genus
  $g$: we uniquely focus on describing the measure $\d \Volwp[\g,\n](\x,Y) / \d \ell$ on the
  level-sets of $\cT_{\g,\n}^*$, and the properties of integrals of the form \eqref{eq:form_exp}.

  It is very important to notice that the functions we integrate have a very specific structure, not
  dissimilar to the Friedman--Ramanujan hypothesis. Actually, we study, in all generality, the
  integrals
  \begin{equation}
    \label{eq:form_exp_gen}
    \int_{\ell_Y(\curve)=\ell}
    \prod_{j \in V} f_j(x_j)
    \prod_{j=1}^k x_{i_j} \delta(x_{i_j}-x_{i_j'})
    \frac{\d \Volwp[\g,\n](\x,Y)}{\d \ell}
  \end{equation}
  where the functions $f_j$ are so that, for all $j \in V$, $f_j(x)e^{x/2}/x$ is a
  Friedman--Ramanujan function. One of the key steps to the proof of \cref{thm:dream} is that, under
  this hypothesis, any integral of the form \eqref{eq:form_exp_gen} is a Friedman--Ramanujan
  function in the weak sense.
  
\subsection{New set of coordinates on Teichm\"uller spaces}
\label{sec:expr-weil-peterss}

Let us now focus on the second challenge raised by \cref{thm:express_average}: the integration on
the level-sets
\begin{equation}
  \{(\x,Y) \in \cT_{\g,\n}^* \; : \; \ell_Y(\curve)=\ell\}
\end{equation}against the
Weil--Petersson volume. This is completely new territory and we have had to develop a lot of tools from
scratch. Notably, we introduce a family of new coordinate-systems on the spaces $\cT_{\g,\n}^*$. 

\subsubsection{Opening the intersections and generalized eights}
\label{sec:our-geom-constr}

Let us consider a local type $\type = \eqc{\Sf,\curve}$. We assume the type is non-simple because
the simple case is already well-known. We pick a representative of the homotopy class of $\curve$ in
minimal position, with all intersections transverse and with only double points. Let $r \geq 1$
denote the number of self-intersections of $\curve$.

\begin{figure}[h!]
  \centering
  \begin{subfigure}[b]{0.5\textwidth}
    \centering
    \includegraphics{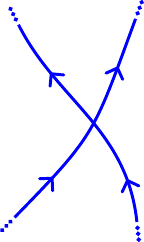}
    \caption{The initial intersection.}
    \label{fig:open_0}
  \end{subfigure}%
  \begin{subfigure}[b]{0.5\textwidth}
    \centering
    \includegraphics{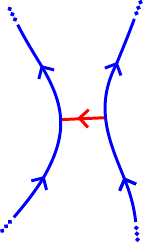}
    \caption{Opening the intersection.}
    \label{fig:open_first}
  \end{subfigure}%
  \caption{Opening an intersection to replace it by a bar. % The orientation of the
    % bars $B_j^\pm$ depends on the type of opening: in the first way, $B_j^+$ and $B_j^-$ are
    % oriented in opposite ways whilst $B_j^+=B_j^-$ in the second way.
  }
  \label{fig:opening_int}
\end{figure}%

We consider a family of disjoint neighborhoods of the $r$ self-intersection points of $\curve$,
such that the restriction of the loop $\curve$ to each neighborhood is a cross, as represented in
\cref{fig:open_0}. In each neighborhood, we \emph{open} the intersection as represented in
\cref{fig:open_first}. We keep track of the presence of the intersection by replacing it with an
oriented bar.

The result of this operation is what we call a \emph{diagram}, a multi-loop $\beta$ together with a
family of $r$ bars $B$ with endpoints on $\beta$, with no self-intersection outside of the endpoints
of the bars. See \cref{fig:open_ex} for some examples.

\begin{figure}[h!]
  \centering
    \includegraphics{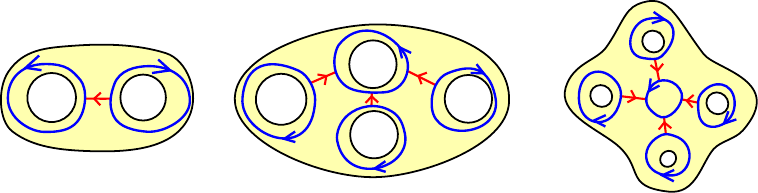}
  \caption{The result of our opening procedure on the examples of \cref{fig:examples_types}.}
  \label{fig:open_ex}
\end{figure}%

Note that the multi-loop $\beta$ is not always a multi-curve: whilst it has no self-intersections,
some of its components can be contractible or homotopic to one another. This happens when there are
disks in the complement of $\curve$, as is the case for instance in the rightmost example of
\cref{fig:open_ex}. We shall be mostly interested in a specific type of loop called generalized
eights, defined below.

\begin{defi}
  We say $\curve$ is a \emph{generalized eight} if no component of the regular neighborhood of
  $\curve$ in $\Sf$ is contractible in $\Sf$.
\end{defi}

This definition has the following consequence.

\begin{lem}
  If $\curve$ is a generalized eight, then $\beta$ is a multi-curve, i.e. all of its components are
  non-contractible and homotopically distinct. The diagram $(\beta, B)$ is uniquely defined, and
  $\chi(\Sf)=r$.  
\end{lem}

Throughout the rest of these notes, we assume that $\curve$ is a generalized eight. This is
the most interesting case in our proof. The two examples on the left of Figure \ref{fig:open_ex} are
generalized eights.

\subsubsection{Geodesic diagram}

Let us now equip the surface $\Sf$ with a metric, i.e. take a point $(\x,Y) \in
\cT_{\g,\n}^*$. Because $\curve$ is a generalized eight, we can pick the representative of $Y$ in
the Teichm\"uller space so that the multi-curve $\beta$ is a multi-geodesic. The bars $B$ are then
paths from $\beta$ to itself, which do not intersect.

We consider homotopy classes with endpoints \emph{gliding} along $\beta$: two paths $\gamma_0$,
$\gamma_1$ on $\Sf$ with endpoints on $\beta$ are said to be homotopic in that sense if there is a
continuous map $[0,1] \times [0,1] \rightarrow \Sf$, coinciding with $\gamma_0$ and $\gamma_1$ when
the first variable is set to be $0$ or $1$ respectively, and with any point of coordinates $(t,0)$
or $(t,1)$ on $\beta$.  If $\beta$ is a multi-geodesic on the metric surface $Y$, then any
non-trivial homotopy class with endpoints gliding along $\beta$ has a unique length-minimizing
element, which is a orthogeodesic with endpoints on $\beta$. On all figures, we will mark the
endpoints of such orthogeodesics by large dots to symbolize the fact that there is a right-angle
there.

\begin{figure}[h]
  \centering
  \includegraphics{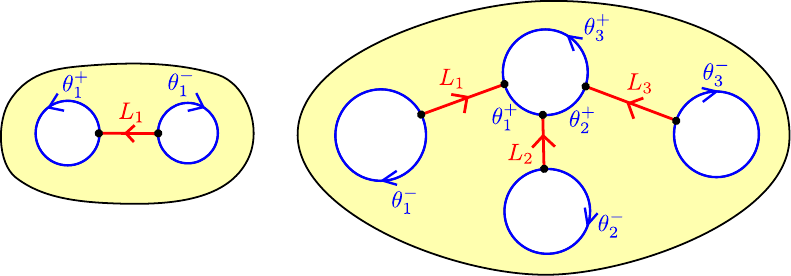}
  \caption{The geodesic diagrams for the generalized eights above.}
  \label{fig:geod_diag}
\end{figure}

We shall replace the bars $B$ by orthogeodesic representatives (here we use the fact that $\curve$
is a generalized eight to check none of these homotopy classes are trivial). We obtain the
\emph{geodesic representative} of the diagram $(\beta, B)$. See \cref{fig:geod_diag} for two
examples. Our new coordinates on the Teichm\"uller space $\cT_{\g,\n}^*$ will be the family of lengths
naturally appearing on this geodesic diagram.

\subsubsection{Definition of our coordinates}
\label{sec:defin-labell-coord}

Let us pick an arbitrary labelling of the $r$ intersection points of $\curve$ by the set
$\{1, \ldots, r\}$. This means that the bars $B$ recording the intersections are now labelled as
$B_1, \ldots, B_r$. Recall that we picked an orientation convention when constructing the bars in
\cref{fig:opening_int}. We shall denote as $B_1^+, \ldots, B_r^+$ the bars oriented as in
\cref{fig:opening_int}, and $B_1^-, \ldots, B_r^-$ the bars with reverse orientation. 

On the topological diagram, just after we open the intersections, we observe that $\beta$ is cut
into $2r$ segments by the bars $B$. We call these segments \emph{simple portions} because they
correspond to maximal sub-segments of $\curve$ with no self-intersections. The simple portions are
oriented in a fashion which is consistent with the initial loop $\curve$, as illustrated in
\cref{fig:opening_int}. We label them by the set
\begin{equation}
  \label{eq:thetadef}
  \Theta := \{1, \ldots, r\} \times \{\pm\}
\end{equation}
by letting, for any $q = (j, \epsilon) \in \Theta$, $\cI_q = \cI_j^\epsilon$ be the simple portion
starting at the terminus of the bar $B_q=B_j^\epsilon$. 

If now we pick a metric on $\Sf$, we shall denote as $\overline{B}_q$ and $\overline{\cI}_q$ the
bars and simple portions corresponding to $B_q$ and $\cI_q$ in the geodesic representative of the
diagram $(\beta,B)$. Note that, as we make the endpoints of $B_q$ glide along $\beta$, the simple
portions $(\cI_q)_{q \in \Theta}$ might change orientation or wind around $\beta$.

We define the following geometric quantities, illustrated in \cref{fig:geod_diag}.

\begin{defi}
  For $j \in \{1, \ldots, r\}$, we denote as $L_j>0$ the length of the orthogeodesic
  $\overline{B}_j$.
  For $q \in \Theta$, we denote as $\theta_q$ the algebraic length of the segment
  $\overline{\cI}_q$, positive if it is oriented as in the initial diagram, negative otherwise.
\end{defi}

Note that the vectors $\vec{L} = (L_1, \ldots, L_r) \in \R_{>0}^r$ and $\vec{\theta} =
(\theta_1^\pm, \ldots, \theta_r^\pm) \in \R^\Theta$ contain $3r$ components. This is the dimension
of the Teichm\"uller space $\cT_{\g,\n}^*$. 

To conclude this section, we observe that the lengths of the multi-geodesic $\beta$ can be written
simply as a sum of $\theta$-parameters. More precisely, we denote as $\Lambdabeta$ the set of
components of $\beta$, so that $\beta = (\beta_\lambda)_{\lambda \in \Lambdabeta}$. For each
$\lambda \in \Lambdabeta$, writing $y_\lambda = \ell_Y(\beta_Y)$, we then have
\begin{equation}
  \label{eq:l_beta}
  y_\lambda
  = \sum_{q \in \Theta_t(\lambda)} \theta_q
  \quad \text{where} \quad
  \Theta_t(\lambda) := \{q \in \Theta \, : \, B_q \text{ terminates on } \beta_\lambda\}.
\end{equation}
(The notation $\lambda \in \Lambdabeta$ is a bit strange here, but we keep it to be consistent with
the notations in \cite{anantharaman2025}.)

\subsubsection{Computation of the Weil--Petersson volume form}
\label{sec:comp-weil-peterss}

\paragraph{Statement}

Our objective in introducing the variables $\vec{L}$, $\vec{\theta}$ is to provide new sets of
coordinates on the space $\cT_{\g,\n}^*$, the Teichm\"uller space of hyperbolic metrics on the
surface $\Sf$, adapted to our loop $\curve$. This space is naturally equipped with the
Weil--Petersson volume form
$\d \mathrm{Vol}^{\mathrm{WP}}_{\g, \n}(\x,Y) = \d \x \d \mathrm{Vol}^{\mathrm{WP}}_{\g, \n}(Y)$. We
prove the following.

\begin{thm}\label{t:maindet}
  For any generalized eight $\mathbf{c}$ filling $\Sf$,
  the map
  \begin{equation}
    \label{eq:ch_var}
    \begin{cases}
      \cT^*_{\g, \n} & \rightarrow \R_{>0}^{r} \times \R^\Theta \\
      (\x, Y) & \mapsto (\vec{L}, \vec{\theta})
    \end{cases}
  \end{equation}
  is a $\mathcal{C}^1$-diffeomorphism onto its image $\domain$,
  and
  \begin{equation}\label{e:maindet}
    \begin{split}
     & 2^{\n}\prod_{j=1}^{\n} \sinh\div{x_j} \d \mathrm{Vol}^{\mathrm{WP}}_{\g, \n}(\x,Y)\\
     & = 2^{\# \Lambdabeta} \prod_{\lambda \in \Lambdabeta} \sinh^2\div{y_\lambda} \prod_{j=1}^r \sinh(L_j) \d^r \vec{L} \d^{2r}
      \vec{\theta}
    \end{split}
  \end{equation}
  where $y_\lambda$ is expressed in \eqref{eq:l_beta}, $\d^r \vec{L}=\prod_{j=1}^r \d L_j$ and
  $\d^{2r} \vec{\theta}=\prod_{q \in \Theta} \d \theta_q$.  % Furthermore, the boundary of
  % $\domain$ is contained in $\bigcup_{\lambda\in \Lambda} \{ y_\lambda=0\}$.
\end{thm}

\begin{rem}
  The result we prove in \cite{anantharaman2025} is actually more general in two ways.
  \begin{itemize}
  \item There is actually two distinct ways to open any given intersections of
    $\curve$. In~\cite{anantharaman2025}, we refer to the choice we made in \cref{fig:opening_int}
    as ``the first way''. It is our preferred resolving of the intersection because it respects the
    orientation of the loop~$\curve$, and this is the reason why I have only presented this way in
    the notes. However, we can pick a way for each intersection, and this actually yields different
    coordinate systems on $\cT_{\g,\n}^*$, for which \eqref{e:maindet} holds.
  \item In all rigour, we need to consider \emph{multi-loops}, i.e. families of loops
    $\curve = (\curve_1, \ldots, \curve_\cc)$ with $\cc$ components filling $\Sf$. I have hidden
    this difficulty to make the exposition lighter, but it is necessary to do so because the proof
    of \cref{t:maindet} by induction requires it. Though sometimes slightly cumbersome in terms of
    notation, this distinction does not generate any difficulties.
  \end{itemize}
\end{rem}

\paragraph{Wolpert's formula and pairs of pants decomposition}

Our proof relies on Wolpert's magic formula for the Weil--Petersson volume in Fenchel--Nielsen
coordinates. These coordinates are generated by decomposing the surface $\Sf$ into pairs of pants,
and considering the lengths that then naturally appear. A pair of pants decomposition is a
multi-curve $\Gamma = (\Gamma_\lambda)_{\lambda \in \Lambda}$ on $\Sf$ such that the complement of
$\Gamma$ is a family of surfaces of signature $(0,3)$.

Here, we consider that the boundary components of $\Sf$ itself are included in the pair of pants
decomposition. We denote as $\Lambdain$ the set of inner components of $\Lambda$, i.e. components of
the pair of pants decomposition which are not boundary components of $\Sf$.

For $\lambda \in \Lambda$, we denote as $y_\lambda$ the geodesic length of $\Gamma_\lambda$. If
$\lambda \in \Lambdain$, i.e. if $\Gamma_\lambda$ is an inner curve, then we pick a twist parameter
$\alpha_\lambda$ along $\Gamma_\lambda$. Then, Wolpert's magic formula states that
\begin{equation}
  \label{eq:FN}
  \d \Volwp[\g,\n](\x,Y)
  = \prod_{\lambda \in \Lambda} \d y_\lambda \prod_{\lambda \in \Lambdain} \d \alpha_\lambda.
\end{equation}
We have a length-parameter $y_\lambda$ for every element of the pair of pants decomposition,
including the boundary components of $\Sf$, because we consider the lengths $\x \in \R_{>0}^{\n}$ to
be a free variable in $\cT_{\g,\n}^*$. To the contrary, we only have a twist for the inner curves.

In order to apply Wolpert's formula, we construct a pair of pants decomposition
$(\Gamma_\lambda)_{\lambda \in \Lambda}$ associated to our loop $\curve$ by considering successively
the bars $B_1, \ldots, B_r$. We do so by observing that a simple bar connecting two (non-necessarily
distinct) simple loops, with no other intersections, generates a pair of pants. This allows to
decompose $\Sf$ into $r$ pairs of pants, corresponding to the $r$ self-intersections of $\curve$.
This decomposition is not canonical and depends on the choice of numbering for the intersection
points of $\curve$.
An example of our pair of pants decomposition is represented in \cref{fig:pop_dec}.
\begin{figure}[h]
  \centering
  \includegraphics{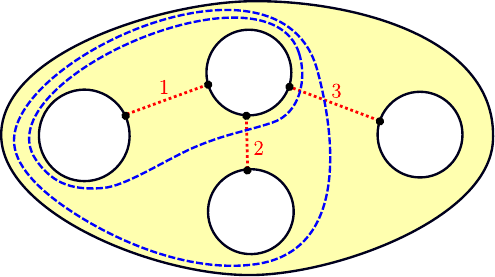}
  \caption{Example of the pair of pants decomposition we construct. The set $\Lambda$ contains seven
    elements: the five boundary components of $\Sf$ and  $2$ additional inner curves.}
  \label{fig:pop_dec}
\end{figure}

\paragraph{Jacobian}

Now that we have defined two sets of coordinates on $\cT_{\g,\n}^*$, we compute the determinant of
the Jacobian of the change of variable from these Fenchel--Nielsen parameters to the coordinates
$(\vec{L},\vec{\theta})$.  

Here are a few elements of the proof.

\begin{proof}
  We prove the result by induction on the number $r$ of self-intersections. There are two cases to
  consider in the case $r=1$: the figure-eight filling a pair of pants, as well as the pair of
  simple loops filling a once-holed torus. The computation is straightforward in both cases. For
  instance, for the pair of pants (represented on the left of \cref{fig:geod_diag}), if we denote as
  $\x = (x_1, x_2, x_3) \in \R_{>0}^3$ the lengths of the three boundary components of the pair of
  pants, we have $x_1=y_1=\theta_1^+$ and $x_2=y_2=\theta_1^-$. Cutting the pair of pants along its
  orthogeodesic $B_1$, we obtain a convex right-angled hexagon, the classic trigonometric formula
  \cite[Theorem 2.4.1 (i)]{buser1992} yields
  \begin{equation}
    \label{eq:tri_hex}
    \cosh (L_1)=\frac{\cosh \div{x_1} \cosh \div{x_2}+\cosh \div{x_3}}{\sinh \div{x_1}\sinh \div{x_2}},
  \end{equation}
  hence $\sinh \div{x_3} \d x_3 =2 \sinh(L_1) \sinh \div{x_1}\sinh \div{x_2}  \d L_1$, which leads
  to the claim.

    \begin{figure}[h]
    \centering
    \includegraphics{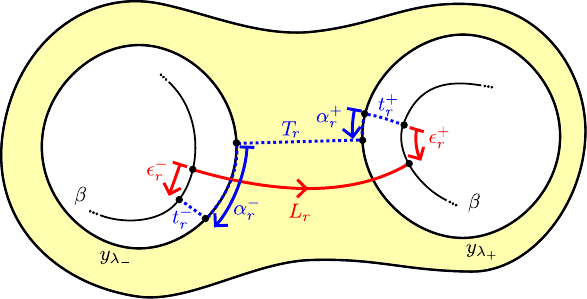}
    \caption{The proof of \cref{t:maindet} in one case. The yellow part is the new pair of pants
      added at the step $r$, which (in this case) connects two distinct boundary components of the
      surface at the step $r-1$. }
    \label{fig:casea}
  \end{figure}

  For $r \geq 2$, we assume the result holds for the same diagram with the bar $B_r$ removed, and
  then work out what adding the bar $B_r$ does to our coordinates. Actually, most coordinates are
  left unchanged, apart from the addition of the new length $L_r$, and the fact that two of the
  prior simple portions are now cut at the two endpoints of $B_r$, and hence become four.  There are
  two cases to consider, depending on the way the last bar is added; one of them is represented in
  \cref{fig:casea}. In both cases, the idea is to break the orthogeodesic $\overline{B}_r$ into a
  succession of perpendicular geodesic segments, in the same homotopy class with endpoints gliding
  along $\beta$. We use these geodesic segments to define the twist parameters, $\alpha_r^\pm$ on
  \cref{fig:casea}. A key observation is that the orthogeodesics outside of the new pair of pants
  are independent of the metric on the pair of pants, and therefore can be viewed as fixed in our
  proof by induction. The computation then reduces to computing the change of variable
  $(T_r,\alpha_r^+,\alpha_r^-) \rightarrow (L_r,\epsilon_r^+,\epsilon_r^-)$ and relating
  $\epsilon_r^\pm$ to our new $\theta$ parameters.
\end{proof}

\begin{rem}
  This approach to proving \cref{t:maindet} does not feel very
satisfactory, as it relies on creating an auxiliary pair of pants decomposition of $\Sf$, which is
done in a quite arbitrary way, and computing the change of variables. We believe it would be
interesting to find a more intrinsic proof based on the definition of the Weil--Petersson form in
terms of quadratic differentials.
\end{rem}

\subsection{Expression for the length function}
\label{sec:expr-length-funct}

Now that we have defined new coordinates on the space $\cT_{\g,\n}^*$ tailored to our generalized
eight $\mathbf{c}$, it is time to take advantage of their structure, and prove a nice formula for
the length of $\mathbf{c}$ in terms of $(\vec{L}$, $\vec{\theta})$.

\subsubsection{Reconstitution of the loop $\curve$}
\label{s:straightening2}

The reason why the coordinates $(\vec{L},\vec{\theta})$ are great to study the length of $\curve$ is
that we can reconstitute $\curve$ by following the bars $B_q$ and simple portions $\cI_q$ for
$q \in \Theta$. More precisely, for $q \in \Theta$, we define $p(q):=B_q \smallbullet \cI_q$ to be
the concatenation of the oriented segments $B_q$ and $\cI_q$ (which do, by definition, connect).
The endpoint of $p(q)$ is the origin of a bar $B_{\sigma q}$ for an element $\sigma q \in
\Theta$. This defines a permutation $\sigma$ of $\Theta$. It is easy to check that $\sigma$ is a
cycle\footnote{Actually, in \cite{anantharaman2025}, $\curve$ is a multi-loop with $\cc$ components,
  in which case $\sigma$ has $\cc$ cycles, each corresponding to a component of $\curve$.} of length
$2r$, and that the loop $\curve$ is freely homotopic to the concatenation
\begin{equation}
  \label{eq:cop}
  \cop := 
  p(q_0) \smallbullet p(\sigma q_0) \smallbullet \ldots \smallbullet p(\sigma^{2r-1}q_0)
\end{equation}
for any given starting point, say for instance $q_0=(1,+)$. 

\begin{figure}[h] \centering
  \includegraphics{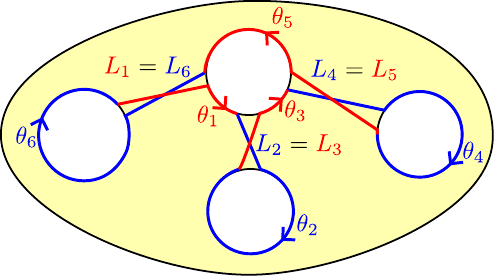}
  \caption{A deformation of the loop $\copg$. The lengths are labelled by the set $\Z_6$ matching
    the order in which the loop is explored, see \cref{sec:relabelling-theta}. The colors represent
    the signs of the bars, red for $+$, blue for $-$.}
  \label{fig:ex_cop}
\end{figure}

Let us now consider a point $Y\in \cT^*_{\g, \n}$. As in \cref{sec:defin-labell-coord}, we pick our
representative in Teichm\"uller space so that $\beta$ is a multi-geodesic. Replacing the bars $B_q$
by their orthogeodesic representatives $\overline{B}_q$ and accordingly adapting the simple portions
to be $\overline{\cI}_q$ allows us to consider a new representative of the homotopy class of
${\mathbf{c}}$, as a cyclic succession of orthogonal geodesics segments:
\begin{align}\label{e:decompo-ig}
  \copg
  = \bar p(q_0)\smallbullet  \ldots \smallbullet \bar p(\sigma^{2r-1}q_0)\end{align}
where $\bar p(q):=\overline{B}_q \smallbullet \overline{\cI}_q$. An example of this is
represented in \cref{fig:ex_cop}.

\subsubsection{Relabelling of $\Theta$}
\label{sec:relabelling-theta}

We shall now relabel the elements of $\Theta$ following the permutation $\sigma$. This can be done
for instance by identifying $\Theta$ with the cyclic set $\Z_{2r} = \Z \diagup 2r \Z$, via the
bijective map
\begin{equation}
  \begin{cases}
    \Z_{2r} & \rightarrow \Theta \\
    k & \mapsto \sigma^{k-1} q_0.
  \end{cases}
\end{equation}
After this identification, the permutation $\sigma$ is simply given by the shift $k \mapsto k+1$.
It is natural to view $\Theta$ as a cyclic set as the writing \eqref{eq:cop} is cyclically
invariant. 

\begin{expl}
  In the example from \cref{fig:ex_cop}, the permutation $\sigma$ is given by
  \begin{equation}
    \label{eq:sigma_expl}
    \sigma = \big((1,+) \quad (2,-) \quad (2,+) \quad (3,-) \quad (3,+) \quad (1,-)\big).
  \end{equation}
\end{expl}

In the rest of this section, following this identification of $\Theta$ with $\Z_{2r}$, for
$k \in \Z_{2r}$, we shall write $B_k$, $\cI_k$ for the $k$-th bar and $k$-th simple portion in the
expression \eqref{eq:cop}, and $p(k)$ for their concatenation, etc. This can lead to slight
conflicts of notations, so we shall always try and clarify which labelling we are using for
$\Theta$.

The orthogonal geodesic segments appearing in $\bar{\mathbf{c}}_{\mathrm{op}}$ have consecutive
lengths $L_1, \theta_1, \ldots, L_{2r}, \theta_{2r}$, where for $k \in \Z_{2r}$, we write $L_k$ for
the length of $\overline{B}_k$ and $\theta_k$ for the algebraic length of $\overline{\cI}_k$. Note
that the vector $(L_k)_{k \in \Z_{2r}}$ contains exactly two copies of each entry of $\vec{L}$ as
each bar is traversed exactly twice.

\subsubsection{Orientation rules}

The signs of the bars will be very important in the following, because they prescribe the directions
of the right turns in $\copg$. Indeed, due to the convention defined in \cref{fig:opening_int}, as
we go along $\copg$, we arrive on each positive bar from the left, and leave it to the right.  The
opposite occurs for negative bars. This is visible on \cref{fig:ex_cop}, where we colored the paths
$\bar{p}(k)$ depending on their signs. These ``orientation rules'' play a key role in
\cite{anantharaman2025}.

For $k \in \Z_{2r}$, we shall denote as $\epsilon_k \in \{\pm\}$ the sign component of
$\sigma^{k-1} q_0$, i.e. the direction in which the $k$-th bar is explored. Then, the geometry of
the geodesic representative $\copg$ is entirely determined by the sequences of lengths $(L_k)_{k \in
\Z_{2r}}$, $(\theta_k)_{k \in \Z_{2r}}$ as well as the signs $(\epsilon_k)_{k \in \Z_{2r}}$, which
tell us which direction to turn at the end of each geodesic segment.

\subsubsection{The formula for the length of $\mathbf{c}$}

We are now ready to state and prove our formula for the length of $\mathbf{c}$ in terms of the
lengths $L_k$, $\theta_k$ and the signs $\epsilon_k$.

\begin{thm}
  \label{prop:form_l_gamma_try}
   For any $Y \in \mathcal{T}_{\g,\n}^*$ of coordinates
  $(\vec{L},\vec{\theta})$, 
  \begin{align*}
    \cosh \div{\ell_Y({\mathbf{c}})}
    & = \sum_{\substack{\delta \in \{ \pm 1 \}^{2r} \\ \delta_1 \ldots \delta_{2r} = +1}}
    \prod_{k=1}^{2r} \hyp_{\delta_k} \div {\theta_k} \cosh
    \left(
    \frac 12
    \sum_{k=1}^{2r} \rho^\delta_k \epsilon_k L_k 
    \right)
  \end{align*}
  where $\rho^\delta_k = \prod_{1 \leq j<k} \delta_{j}$ (for $k=1$ this empty product is equal to
  $1$).
\end{thm}

\begin{expl}
  For $r=1$, the formula reads
  \begin{equation*}
    \cosh \div{\ell_Y(\curve)}
    = \cosh \div{\theta_1^+}\cosh \div{\theta_1^-}
    + \sinh \div{\theta_1^+}\sinh \div{\theta_1^-} \cosh(L_1)
  \end{equation*}
  which is equivalent to our known formula for the figure-eight
  \begin{equation*}
    \cosh \div{\ell_Y(\curve)}
    = 2 \cosh \div{x_1}\cosh \div{x_2}
    + \cosh\div{x_3}
  \end{equation*}
  since $\theta_1^-=x_1$, $\theta_1^+=x_2$ and
  $
    \cosh(L_1) = \frac{\cosh \div{x_1} \cosh \div{x_2}
    + \cosh \div{x_3}}{\sinh \div{x_1} \sinh \div{x_2}}$
by the usual hyperbolic identities in right-angled hexagons.
\end{expl}

\begin{rem}
  We actually prove this result holds for any choice of openings of the intersections
  of~$\mathbf{c}$, which provides different expressions for the length of~$\mathbf{c}$ in different
  coordinate systems. The formula also holds for multi-loops
  $\curve = (\curve_1, \ldots, \curve_\cc)$, in which case we prove it component by component, and
  the lengths appearing for each component are those from the corresponding cycle of $\sigma$.
\end{rem}

\begin{proof}
  It is a classical fact that $\cT^*_{\g, \n}$ may be seen as the space of discrete faithful
  representations of the fundamental group $\pi_1(\mathbf{S})$ into $\mathrm{PSL}(2, \R)$ modulo
  conjugacy. Moreover, if $Y\in \cT^*_{\g, \n}$ corresponds to the representation
  $\rho : \pi_1(\mathbf{S}) \rightarrow \mathrm{PSL}(2, \R)$, we can express all lengths thanks to
  the trace of $\rho$:
  \begin{align} \label{e:trace_cosh}2 \cosh \div{\ell_Y(\gamma)}= |\Tr\, \rho(\gamma)|,\end{align}
  where $\rho(\gamma)\in \mathrm{PSL}(2, \R)$ is the image under $\rho$ of the homotopy class of
  $\gamma$.  Our method to express lengths in the new coordinates consists in providing an explicit
  expression for $\rho(\gamma)$ in~\eqref{e:trace_cosh}, by writing a representative of the homotopy
  class of $\gamma$ as a succession of orthogonal geodesic segments.

We recall that the unit tangent bundle $T^1 \IH$ itself can be identified with
$\mathrm{PSL}(2, \R)$. We can then identify different geometric actions on $T^1\IH$ with the
right-multiplication by a specific matrix: $a^t$ is the action of the geodesic flow at time $t$,
$w^t$ the conjugate of the geodesic flow by a rotation of angle $\pi/ 2$, and $k^\theta$ the
rotation of angle $\theta$ by $k^\theta$, where
\begin{align*}
  a^t =
  \begin{pmatrix}
    e^{\frac{ t}{2}} & 0 \\ 0 & e^{-\frac{t}{2}}
  \end{pmatrix}
  \quad
   w^t= 
  \begin{pmatrix}
    \cosh \div{t} & \sinh \div{t} \\ \sinh \div{t} & \cosh \div{t}
      \end{pmatrix}
  \quad
 k^\theta= 
  \begin{pmatrix}
    \cos \div{\theta} & \sin \div{\theta} \\- \sin \div{\theta} & \cos \div{\theta}
      \end{pmatrix}.
\end{align*}
      These moves are represented on Figure \ref{fig:flows}.

      \begin{figure}[h!]
        \center
     \includegraphics{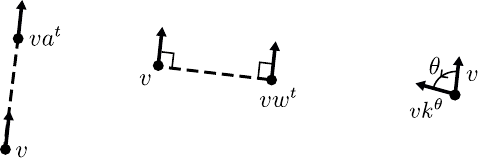}
     \caption{The moves corresponding to the matrices $a^t$, $w^t$ and $k^\theta$.}
     \label{fig:flows}
   \end{figure}

   Using the decomposition \eqref{e:decompo-ig} of $\copg$ as a succession of orthogonal geodesics,
   we prove that
\begin{equation}
  \label{eq:l_gamma_product_trace}
  \cosh \div{\ell_Y({\mathbf{c}})} = \frac{1}{2} \,  \Tr \big(a_1b_1 \ldots a_{2r}b_{2r}\big)
\end{equation}
where the matrices $a_k, b_k$ are defined by
\begin{align*} 
  &a_k = a^{\epsilon_k L_k}=
  \begin{pmatrix}
    e^{\frac{\epsilon_k L_k}{2}} & 0 \\ 0 & e^{-\frac{\epsilon_k L_k}{2}}
  \end{pmatrix} 
   &b_k = w^{\theta_k}=
  \begin{pmatrix}
    \cosh \div{\theta_k} & \sinh \div{\theta_k} \\ \sinh \div{\theta_k} & \cosh \div{ \theta_k}
  \end{pmatrix}.
\end{align*}
The main difficulty in the proof is to remove the absolute value in \eqref{e:trace_cosh}, which we
do by a continuity argument, using our orientation rules.  We then expand the product of matrices,
and conclude by an explicit computation.
\end{proof}

\subsection{Conclusion}

Let us now conclude and provide a few elements we use to prove \cref{thm:FR_type}.  So far, thanks
to our integral formula and asymptotic expansions, we have expressed the volume functions
$V_g^\type(\ell)$ as a combination of functions of the form
\begin{equation}
  \int_{\ell_Y(\curve)=\ell}
  \prod_{j \in V} f_j(x_j)
  \prod_{j=1}^k x_{i_j} \delta(x_{i_j}-x_{i_j'})
  \frac{\d \Volwp[\g,\n](\x,Y)}{\d \ell}
\end{equation}
where the functions $f_j$ are such that, for all $j \in V$, $f_j(x)e^{x/2}/x$ is a Friedman--Ramanujan
function.

We now use our change of variable from $(\x, Y) \in \cT_{\g,\n}^*$ to our adapted variables
$(\vec{L},\vec{\theta}) \in \mathfrak{D} \subset \R_{>0}^r \times \R^\Theta$.  Doing so has two
consequences: we have to plug in the Jacobian of the change of variable (which we computed in
\cref{t:maindet}), and we have to express everything in terms of the new variables
$(\vec{L},\vec{\theta})$. This requires to express the length of the boundary components of $\Sf$,
the $(y_\lambda)_{\lambda \in \Lambda \setminus \Lambdain}$, as a function of
$(\vec{L},\vec{\theta})$, which we do by a method similar to the one in
\cref{sec:expr-length-funct}.

Once we work everything out, we end up with a pseudo-convolution of Friedman--Ramanujan functions,
or, more precisely, an integral of the form
\begin{equation*}
  \int_{\vec{L}, \thetaNe} \psi(\vec{L},\thetaNe)
  \int_{\thetaAc: \ell_Y(\curve)=\ell} \varphi_{\vec{L},\thetaNe}(\thetaAc) \prod_{q \in \Theta_{\mathrm{ac}}} \theta_q^{\rK_q} e^{\theta_q}
  \frac{\d^r \vec{L} \d^{2r} \vec{\theta}}{\d \ell}
\end{equation*}
where we have split the set of parameters
$\Theta = \Theta_{\mathrm{ac}} \sqcup \Theta_{\mathrm{ne}}$ into \emph{active} and \emph{neutral}
parameters and accordingly the vector $\vec{\theta}$ as
$\thetaAc = (\theta_q)_{q \in \Theta_{\mathrm{ac}}}$ and
$\thetaNe = (\theta_q)_{q \in \Theta_{\mathrm{ne}}}$. The active parameters
$q \in \Theta_{\mathrm{ac}}$ are the terms which have a principal term in the integral above,
$\theta_q^{\rK_q}e^{\theta_q}$, whilst the neutral parameters do not (they are Friedman--Ramanujan
remainders). Interestingly, all the parameters $\vec{L}$ will always be neutral parameters as they
do not contribute towards principal terms\footnote{The reason why this is true is that each bar
  $B_j$ is explored twice in the trajectory of $\curve$. This is a general fact: if one simple
  portion is explored twice as we go along $\curve$, it will correspond to neutral parameters (we
  call these simple portions \emph{shielded} in \cite{anantharaman2025}).}.

We then consider all the neutral variables as constants, and focus on the pseudo-convolution 
\begin{equation}
  \label{eq:pseu_c_1}
  \int_{\thetaAc: \ell_Y(\curve)=\ell}
  \varphi_{\vec{L},\thetaNe}(\thetaAc) \prod_{q \in \Theta_{\mathrm{ac}}} \theta_q^{\rK_q} e^{\theta_q}
  \frac{\d \thetaAc}{\d \ell}.
\end{equation}
This pseudo-convolution has the following defining functions.
\begin{itemize}
\item The height function is the function $\thetaAc \mapsto \ell_Y(\curve)$, which we have expressed
  in \cref{sec:expr-length-funct}.
\item The weight function $\varphi_{\vec{L},\thetaNe}(\thetaAc)$ contains a lot of defects coming
  from the expansions and substitutions we made in the computation. It notably contains the
  indicator function of our domain $\mathfrak{D}$ as well as terms coming from expressing
  $y_\lambda$ as a function of $(\vec{L},\vec{\theta})$.
\end{itemize}
At this point, we wish to reproduce the proof of \cref{t:intermediate} and apply the operator $\cL$
to \eqref{eq:pseu_c_1}. We will want to use that the pseudo-convolution is close to an actual
convolution, i.e. the height function is close to $\sum_{q \in \Theta_{\mathrm{ac}}}\theta_q$ and the
  weight has small derivatives. However, there are two challenges here.
\begin{enumerate}
\item The derivatives of the weight function $\varphi_{\vec{L},\thetaNe}$ are difficult to estimate,
  notably due to the discontinuity of the indicator function, and the terms coming from the
  substitution of $y_\lambda$ as a function in $(\vec{L},\vec{\theta})$ during our change of
  variable.
\item Whilst one can check that the height function does belong to the class $\cE_n^{(a)}$ directly
  from the expression of the length in the coordinates $(\vec{L},\vec{\theta})$, which is what we
  want, we actually cannot, in all generality, force the parameters $\thetaAc$ to live in an
  interval of the form $(a,+\infty)^n$ for $a>0$ on the whole domain of integration. This is due to
  the fact that some of the $\theta_q$ might take negative values in certain geometric
  configurations (which we call \emph{crossings}).
\end{enumerate}
We solve both these technical issues by introducing two partitions of unity of $\cT_{\g,\n}^*$, one
for each issue, which allows us to understand better which lengths and parameters do or do not go to
infinity. We  can then understand the derivatives of the weight function.  On each element of
the partition, we introduce new modified coordinates $\tilde{\theta}_q$, so that the height function
actually satisfies an hypothesis $\cE_n^{(a)}$ on the whole space.

Once all of this work is completed, we conclude by using a refined version of
\cref{t:intermediate}, relying on a geometric comparison estimate. Once we have applied the
operator $\cL$ enough times to ``kill'' all the exponential terms $\theta_q^{\rK_q}e^{\theta_q}$, we
then integrate against all the neutral parameters, and conclude that the integral is a
Friedman--Ramanujan remainder. 

% \section{Tangles and the Moebius inversion formula}
% \label{cha:tangl-moeb-invers}

% \subsection{The issue caused by tangles}
% \label{sec:issue-caused-tangles}

% \subsection{Inclusion-exclusions and the Moebius function}
% \label{sec:incl-excl-moeb}

\bibliographystyle{alpha}
\bibliography{bibliography}

\begin{thebibliography}{CGTvH25}

\bibitem[Alo86]{alon1986}
Noga Alon.
\newblock Eigenvalues and expanders.
\newblock {\em Combinatorica}, 6(2):83--96, 1986.

\bibitem[AM22]{anantharaman2022}
Nalini Anantharaman and Laura Monk.
\newblock A high-genus asymptotic expansion of {{Weil}}--{{Petersson}} volume
  polynomials.
\newblock {\em Journal of Mathematical Physics}, 63(4):043502, 2022.

\bibitem[AM23]{anantharaman2023}
Nalini Anantharaman and Laura Monk.
\newblock Friedman-{{Ramanujan}} functions in random hyperbolic geometry and
  application to spectral gaps.
\newblock {\em arXiv:2304.02678}, April 2023.

\bibitem[AM24]{anantharaman2024}
Nalini Anantharaman and Laura Monk.
\newblock Spectral gap of random hyperbolic surfaces.
\newblock {\em arXiv:2403.12576}, March 2024.

\bibitem[AM25]{anantharaman2025}
Nalini Anantharaman and Laura Monk.
\newblock Friedman-{{Ramanujan}} functions in random hyperbolic geometry and
  application to spectral gaps {{II}}.
\newblock {\em arXiv:2502.12268}, February 2025.

\bibitem[Bon22]{bonifacio2022}
James Bonifacio.
\newblock Bootstrapping closed hyperbolic surfaces.
\newblock {\em Journal of High Energy Physics}, 2022(3):93, March 2022.

\bibitem[Bor20]{bordenave2020}
Charles Bordenave.
\newblock A new proof of {{Friedman}}'s second eigenvalue theorem and its
  extension to random lifts.
\newblock {\em Annales Scientifiques de l'{\'E}cole Normale Sup{\'e}rieure.
  Quatri{\`e}me S{\'e}rie}, 53(6):1393--1439, 2020.

\bibitem[Bus82]{buser1982}
Peter Buser.
\newblock A note on the isoperimetric constant.
\newblock {\em Annales Scientifiques de l'{\'E}cole Normale Sup{\'e}rieure.
  Quatri{\`e}me S{\'e}rie}, 15(2):213--230, 1982.

\bibitem[Bus84]{buser1984}
Peter Buser.
\newblock On the bipartition of graphs.
\newblock {\em Discrete Applied Mathematics}, 9(1):105--109, September 1984.

\bibitem[Bus92]{buser1992}
Peter Buser.
\newblock {\em Geometry and {{Spectra}} of {{Compact Riemann Surfaces}}}.
\newblock Birkh{\"a}user, Boston, 1992.

\bibitem[CGTvH25]{chen2025}
Chi-Fang Chen, Jorge {Garza-Vargas}, Joel~A. Tropp, and Ramon van Handel.
\newblock A new approach to strong convergence.
\newblock {\em arXiv:2405.16026}, February 2025.

\bibitem[Che70]{cheeger1970}
Jeff Cheeger.
\newblock A lower bound for the smallest eigenvalue of the {{Laplacian}}.
\newblock In {\em Problems in Analysis}, pages 195--199. Princeton University
  Press, 1970.

\bibitem[Fri91]{friedman1991}
Joel Friedman.
\newblock On the second eigenvalue and random walks in random d-regular graphs.
\newblock {\em Combinatorica}, 11(4):331--362, 1991.

\bibitem[Fri03]{friedman2003}
Joel Friedman.
\newblock A proof of {{Alon}}'s second eigenvalue conjecture.
\newblock {\em Proceedings of the thirty-fifth annual ACM symposium on Theory
  of computing}, pages 720--724, 2003.

\bibitem[GK19]{golubev2019}
Konstantin Golubev and Amitay Kamber.
\newblock Cutoff on hyperbolic surfaces.
\newblock {\em Geometriae Dedicata}, 203(1):225--255, 2019.

\bibitem[GS97]{graaf1997}
Maurits Graaf and Alexander Schrijver.
\newblock Making {{Curves Minimally Crossing}} by {{Reidemeister Moves}}.
\newblock {\em Journal of Combinatorial Theory, Ser. B}, 70:134--156, 1997.

\bibitem[HM23]{hide2023a}
Will Hide and Michael Magee.
\newblock Near optimal spectral gaps for hyperbolic surfaces.
\newblock {\em Annals of Mathematics}, 198(2), September 2023.

\bibitem[HMY25]{huang2025}
Jiaoyang Huang, Theo McKenzie, and Horng-Tzer Yau.
\newblock Ramanujan {{Property}} and {{Edge Universality}} of {{Random Regular
  Graphs}}.
\newblock {\em arXiv:2412.20263}, February 2025.

\bibitem[Hub74]{huber1974}
Heinz Huber.
\newblock {{\"U}ber den ersten Eigenwert des Laplace-Operators auf kompakten
  Riemannschen Fl{\"a}chen}.
\newblock {\em Commentarii mathematici Helvetici}, 49:251--259, 1974.

\bibitem[Kim03]{kim2003}
Henry~H. Kim.
\newblock Functoriality for the exterior square of {{GL}}{\textsubscript{4}}
  and the symmetric fourth of {{GL}}{\textsubscript{2}}.
\newblock {\em Journal of the American Mathematical Society}, 16(1):139--183,
  2003.

\bibitem[KMP24]{kravchuk2024}
Petr Kravchuk, Dalimil Mazac, and Sridip Pal.
\newblock Automorphic {{Spectra}} and the {{Conformal Bootstrap}}.
\newblock {\em Communications of the American Mathematical Society},
  4(1):1--63, January 2024.

\bibitem[LPS88]{lubotzky1988}
A.~Lubotzky, R.~Phillips, and P.~Sarnak.
\newblock Ramanujan graphs.
\newblock {\em Combinatorica}, 8(3):261--277, 1988.

\bibitem[LW24]{lipnowski2024}
Michael Lipnowski and Alex Wright.
\newblock Towards optimal spectral gaps in large genus.
\newblock {\em The Annals of Probability}, 52(2):545--575, March 2024.

\bibitem[Mag20]{magee2020b}
Michael Magee.
\newblock Letter to {{Bram Petri}}, 2020.

\bibitem[McK70]{mckean1970}
Henry~P. McKean.
\newblock An upper bound to the spectrum of {{$\Delta$}} on a manifold of
  negative curvature.
\newblock {\em Journal of Differential Geometry}, 4:359--366, 1970.

\bibitem[Mir07]{mirzakhani2007}
Maryam Mirzakhani.
\newblock Simple geodesics and {{Weil--Petersson}} volumes of moduli spaces of
  bordered {{Riemann}} surfaces.
\newblock {\em Inventiones Mathematicae}, 167(1):179--222, 2007.

\bibitem[Mir13]{mirzakhani2013}
Maryam Mirzakhani.
\newblock Growth of {{Weil--Petersson}} volumes and random hyperbolic surfaces
  of large genus.
\newblock {\em Journal of Differential Geometry}, 94(2):267--300, 2013.

\bibitem[MP19]{mirzakhani2019}
Maryam Mirzakhani and Bram Petri.
\newblock Lengths of closed geodesics on random surfaces of large genus.
\newblock {\em Commentarii Mathematici Helvetici}, 94(4):869--889, 2019.

\bibitem[MPvH25]{magee2025}
Michael Magee, Doron Puder, and Ramon van Handel.
\newblock Strong convergence of uniformly random permutation representations of
  surface groups.
\newblock {\em arXiv:2504.08988}, April 2025.

\bibitem[MZ15]{mirzakhani2015}
Maryam Mirzakhani and Peter Zograf.
\newblock Towards large genus asymptotics of intersection numbers on moduli
  spaces of curves.
\newblock {\em Geometric and Functional Analysis}, 25(4):1258--1289, 2015.

\bibitem[Nil91]{nilli1991}
A.~Nilli.
\newblock On the second eigenvalue of a graph.
\newblock {\em Discrete Mathematics}, 91(2):207--210, 1991.

\bibitem[Rat87]{ratner1987}
Marina Ratner.
\newblock The rate of mixing for geodesic and horocycle flows.
\newblock {\em Ergodic Theory and Dynamical Systems}, 7(2):267--288, 1987.

\bibitem[Sel56]{selberg1956}
Atle Selberg.
\newblock Harmonic analysis and discontinuous groups in weakly symmetric
  {{Riemannian}} spaces with applications to {{Dirichlet}} series.
\newblock {\em The Journal of the Indian Mathematical Society}, 20:47--87,
  1956.

\bibitem[Sel65]{selberg1965}
Atle Selberg.
\newblock On the estimation of {{Fourier}} coefficients of modular forms.
\newblock In {\em Proc. {{Sympos}}. {{Pure Math}}., {{Vol}}. {{VIII}}}, pages
  1--15. Amer. Math. Soc., Providence, R.I., 1965.

\bibitem[SU13]{strohmaier2013}
Alexander Strohmaier and Ville Uski.
\newblock An {{Algorithm}} for the {{Computation}} of {{Eigenvalues}},
  {{Spectral Zeta Functions}} and {{Zeta-Determinants}} on {{Hyperbolic
  Surfaces}}.
\newblock {\em Communications in Mathematical Physics}, 317(3):827--869,
  February 2013.

\bibitem[WX22]{wu2022}
Yunhui Wu and Yuhao Xue.
\newblock Random hyperbolic surfaces of large genus have first eigenvalues
  greater than 3/16-{$\varepsilon$}.
\newblock {\em Geometric and Functional Analysis}, 32(2):340--410, 2022.

\end{thebibliography}

\end{document}